\documentclass[preprint]{imsart}

\usepackage{amsthm,amsmath}
\usepackage{amsfonts}%
\usepackage{amssymb}%
\usepackage{dsfont}
\usepackage{graphicx}

\usepackage{subfigure} 

\startlocaldefs

\newtheorem{Theorem}{Theorem}[section]
\newtheorem{Proposition}[Theorem]{Proposition}

\newtheorem{Lemma}[Theorem]{Lemma}
\newtheorem{Definition}[Theorem]{Definition}
\newtheorem{Notation}[Theorem]{Notation}
\newtheorem{Remark}[Theorem]{Remark}

\renewcommand{\L}{\mathcal{L}}

\newcommand{\E}{\mathbb{E}}

\newcommand{\R}{\mathbb{R}}

\newcommand{\Z}{\mathbb{Z}}
\newcommand{\1}{\mathds{1}}

\endlocaldefs

\begin{document}

\begin{frontmatter}

\title{On a Voter model on $\R^{\lowercase{d}}$:\\Cluster growth in the\\ Spatial $\Lambda$-Fleming-Viot Process}
\runtitle{Cluster growth in the S$\Lambda$FV Process}


\author{\fnms{Habib} \snm{Saadi}\ead[label=e1]{h.saadi@imperial.ac.uk} \thanksref{t1} \thanksref{t2}}
\affiliation{Imperial College London}

\thankstext{t1}{Supported by EPSRC Grant EP/E065945/1.}
\thankstext{t2}{The author would like to thank Alison Etheridge, Nic Freeman and Mladen Savov for carefully reading earlier versions of this manuscript.}

\runauthor{H. Saadi}

\begin{abstract}
The spatial $\Lambda$-Fleming-Viot (S$\Lambda$FV) process introduced in 
(Barton, Etheridge and V\'eber, 2010)
can be seen as a direct extension of the Voter Model (Clifford and Sudbury, 1973); (Liggett, 1997).
As such, it is an Interacting Particle System with configuration space $\mathcal{M}^{\R^d}$, where $\mathcal{M}$ is the set of probability measures on some space $K$. Such processes are usually studied thanks to a dual process that describes the genealogy of a sample of particles. In this paper, we propose two main contributions in the analysis of the S$\Lambda$FV process. The first is the study of the growth of a cluster, and the suprising result is that with probability one, every bounded cluster stops growing in finite time. In particular, we discuss why the usual intuition is flawed. The second contribution is an original method for the proof, as the traditional (backward in time) duality methods fail. 
We develop a forward in time method that exploits a martingale property of the process. To make it feasible, we construct adequate objects that allow to handle the complex geometry of the problem. We are able to prove the result in any dimension $d$.
\end{abstract}



\end{frontmatter}


	\section{Introduction}

				The spatial $\Lambda$-Fleming-Viot process (S$\Lambda$FV) is a model 
				used
				to represent biological evolution on a continuum.
				It was first 
				introduced in \cite{ETHERIDGE_2008_DDASSMMOE_BCP}, 
				and then studied in more details in \cite{BARTON_ETHERIDGE_KELLEHER_2010} , 
				\cite{BARTON_ETHERIDGE_VEBER_2010} and \cite{BERESTYCKI_ETHERIDGE_VEBER_2011}.
				In this setting, given a set of genetic types $K$, a population 
				living on $\R^d$ is represented by a collection of probability measures on $K$. 
				More precisely, the genetic composition at time $t$ of the population at point $x \in \mathbb{R}^d$ 
				is given by a measure $\rho_t(x,\cdot)$ on the type space $K$. 
		        The S$\Lambda$FV  process
   				is a direct spatial extension of the generalised Fleming-Viot processes presented in \cite{DONNELLY_KURTZ_1999a}  and 
   				studied in \cite{BERTOIN_LEGALL_2003}.
		        But it can also be seen as an 
		        interacting particle system generalising 
		        the Voter Model \cite{CLIFFORD_SUDBURY_1973,LIGGETT_1997}.
		        The configuration space for the Voter Model is $\{0,1\}^{\Z^d}$, whereas for the S$\Lambda$FV  process
		        it is $\mathcal{M}^{\R^d}$, where $\mathcal{M}$ is the
		        set of probability measures on $K$. This generalisation of the configuration space is one of the elements
		        that make the study of the S$\Lambda$FV process particularly challenging.\\
		        
		        Our motivation in this article is the study of the fate of a new genetic type created by mutation at time
		        $0$. More precisely, we assume that there are only two types of individuals, $Blue$ and $Red$,
		        and that the new type, say $Red$, occupies a bounded set of $\R^d$ at time $0$. 
				The question is how far this newly created type is going to spread. 
														        
				Because we are working with two types only, the setting simplifies. We have 
				$ \rho_t(x,Blue)=1-\rho_t(x,Red)$, so at time $t$ it is enough to consider the collection of 
				numbers 
				$ \{\rho_t(y,Red), \, y \in \R^d \}$.
				This is why we are going to represent the population at time $t$ by the function
				$$ X_t : \R^d\mapsto [0,1]$$ 
				such that $X_t(y)=\rho_t(y,Red)$. Working with a function instead of a collection of
				probability measures allows us to simplify the notation when manipulating the 
				S$\Lambda$FV process.

		\subsection{The process}
														
				For every time $t \geq 0$, let $X_t$ be a function from $\R^d$ to $[0,1]$. The quantity $X_t(y)$ for $y \in \R^d $ is the 
				frequency of $Red$ individuals at location $y$. The dynamics of $X_t$ is the following.
				Consider a space-time Poisson point process
				$\Pi$  on $\mathbb{R}_+ \times \mathbb{R}^d \times [0,1]$  with rate $dt\otimes dc \otimes dv$, and two constants 
				$0 \leq U < 1$ and $R>0$.
		        Then, for every point $(t,c,v)$ of $\Pi$,
		  		\begin{enumerate}
			  		\item[i)] Draw a ball $B(c,R)$ of radius $R$ centred around 
			  		$c$.
					\item[ii)] If $X_{t-}(c)\geq v$, then the parent is 
					$Red$, and
						for every point $y \in B(c,R)$,
						\begin{equation*}
							X_t(y)=(1-U)X_{t-}(y)+U.
						\end{equation*}
					\item[iii)] If $X_{t-}(c)< v$, then the parent is 
					$Blue$, and
						for every point $y \in B(c,R)$,
						\begin{equation*}
							X_t(y)=(1-U)X_{t-}(y).
						\end{equation*}
				\end{enumerate}
				The steps i), ii) and iii) can be written in a single equation:
				
				\begin{align}
				\label{Single_Transition}
					 X_{t}(y)=X_{t-}(y)+U \1_{\{\|y-c\| \leq R\} } \left(\1_{\{v \leq X_{t-}(c)\} }-X_{t-}(y)\right), \quad y \in \R^d.
					 \\
					 \nonumber
				\end{align}
		  		  		  				  		
				In biological terms, each point $t,c,R$ of the Poisson Point process corresponds to a \emph{reproduction 
				event}
				taking place at time $t$ in a ball $B(c,R)$. First, a parent is chosen at random at location $c$. The parent is $Red$ with probability $X_{t-}(c)$ and $Blue$ with probability $1-X_{t-}(c)$, and her offspring are going to have the same type as her.
				Second, competition for finite resources causes a proportion $U$ of the population inside the ball of 
				centre $c$ and radius $R$ to die. Finally, the offspring of the parent replaces the proportion $U$ 
				of individuals who have died. Births and deaths take place simultaneously at time $t$.
				Figure \ref{Fig_Def_Xt} illustrates the births and deaths events taking place during a single transition at time $t$ corresponding
				to the point $(t,c,v)$ from the Poisson point process $\Pi$.
				
				\begin{figure}
					\subfigure[Sampling of the parent]
					{
						\includegraphics{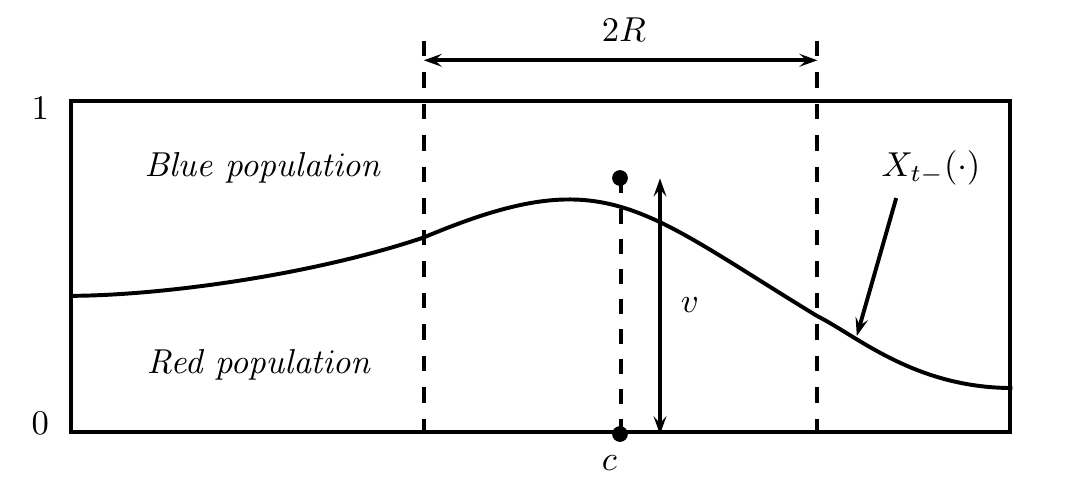}
					}
					\subfigure[Deaths]
					{
						\includegraphics{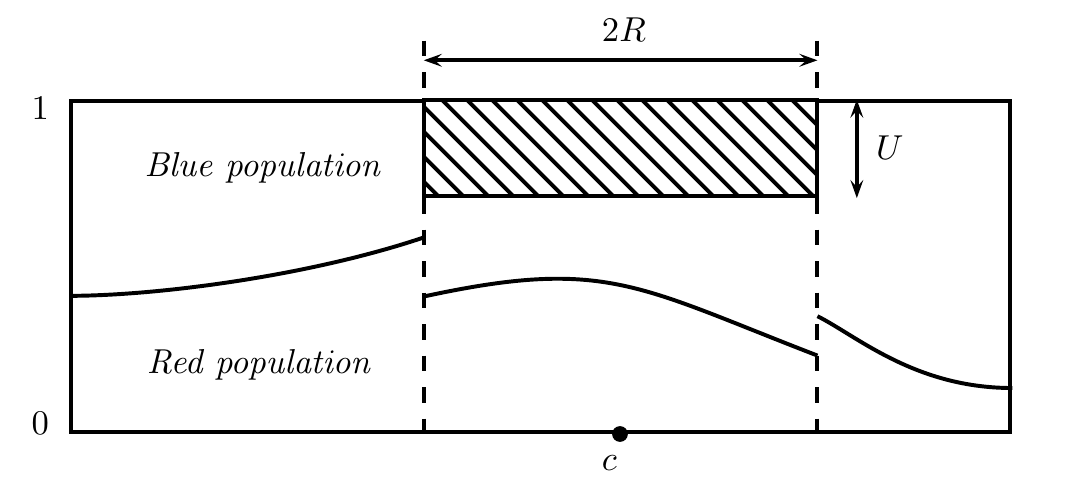}
					}
					\subfigure[Births]
					{
						\includegraphics{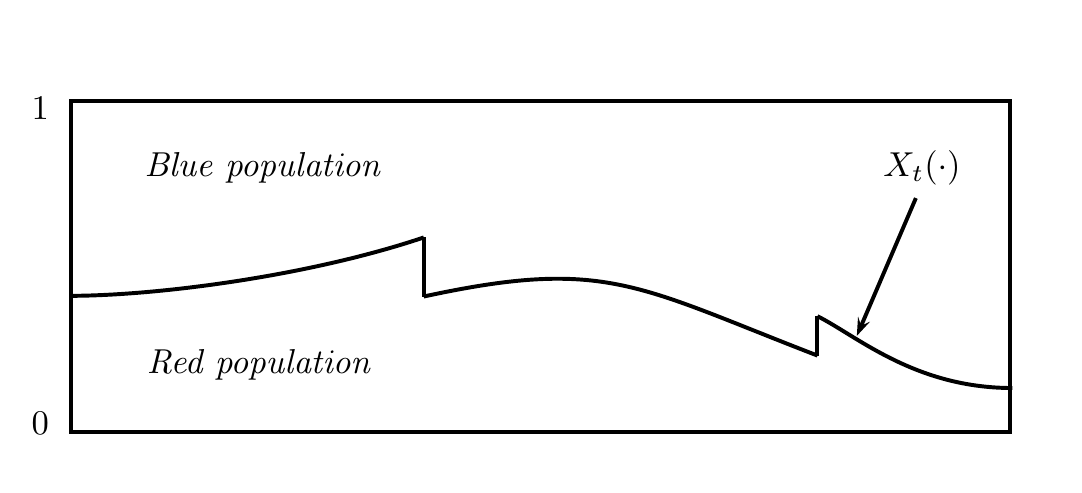}
					}
					\caption{Schematic view in dimension $d=1$ of a Markov transition for the process $X_t(\cdot)$ induced by the point $(t,c,v) \in \Pi$. (a) 
					Because $v > X_{t-}(c)$, the parent is $Blue$. (b) Deaths shrink both populations within the ball $B(c,R)$ by the same factor $1-U$. (c)
					The offspring of the $Blue$ parent replenish the population. All steps (a,b,c) take place instantaneously at time $t$.}
					\label{Fig_Def_Xt}
				\end{figure}

				\begin{Remark}
					We have chosen the parent to be at location $c$, which is a simplification of the model in \cite{BARTON_ETHERIDGE_VEBER_2010}, where the 
					location of the parent was chosen uniformly on the ball. This does not change the model significantly,
					it just simplifies some calculations.
				\end{Remark}
						  				  		  
		  		The presentation of the process we just gave is simply 
		  		an algorithm that describes the jumps of 
		  		$(X_t)_{t\geq0}$, 
		  		but we need to construct it formally as a Markov 
		  		process.
		  		The most natural way is to translate this algorithm
		  		into the infinitesimal generator $\mathcal{L}$ of $X(t)$, which is defined by
		  		\begin{align}
		  			\mathcal{L}I(f) := \lim_{t \rightarrow 0} \frac{\E[ I(X_t) - I(X_0) \, | \, X_0=f]}{t},
		  		\end{align}
		  		where $I$ is a test function, and $f$ is the initial value of the process $(X_t)_{t\geq0}$.
		  		We choose the test function $I$
		  		from the family $I_n( \, \cdot \,;\psi)$ of functions
				of the form
			  	\begin{align}
				  	\label{Test_Func_Red}
		  		  	I_n \big( f;\psi \big)
		  		  	=\int_{(\R^d)^n} \psi(x_1,\dots,x_n) \prod_{i=1}^n
		  		  	 \, f(x_i) \,
		  		  	dx_1 \dots dx_n,
	  		   \end{align}
				where $\psi$ is a function from $(\R^d)$ to $\R$ such that
	  		   $\int |\psi(x_1,\dots,x_n)| dx_1 \dots dx_n< \infty$, and $f$ is a function
	  		   from $(\R^d)^n$ to $[0,1]$ corresponding to $X_0$. The intuition behind this form is that the distribution
	  		   of the function-valued process $(X_t)_{t\geq0}$ is described by the finite-dimensional dynamics at all locations $x_1,\dots,x_n$.

	  		   The generator $\L$ of the process $(X_t)_{t \geq 0}$ is given by

	  		   \begin{align}
		  		  	\nonumber
		  		  	& \quad
		  		  	\L I_n(f;\psi)
		  		  	\\[7pt]
		  		  	\nonumber
		  		  	=
		  		  	\quad &
		  		  	\int_{\R^d}
		  		  	\int_{(\R^d)^n}
		  		  	\sum_{I \subset \{1,\dots,n\}}
		  		  	\Big(
		  		  	\prod_{j \notin I}
		  		  	\1_{x_j \notin B(c,R)}
		  		  	f(x_j)
		  		  	\Big)
		  		  	\times
		  		  	\Big(
		  		  	\prod_{i \in I} \1_{x_i \in B(c,R)}
		  		  	\Big)
		  		  	\\[7pt]
		  		  	\nonumber
		  		  	&
		  		  	\
		  		  	\times
		  		  	\Bigg[
		  		  	f(c) \, 
		  		  	\bigg(
		  		  	\prod_{i \in I} \Big( (1-U) f(x_i)+U \Big)
		  		  	-
		  		  	\prod_{i \in I} f(x_i)
		  		  	\bigg)
		  		  	\\[7pt]
		  		  	\nonumber
		  		  	&
		  		  	\quad \quad \quad \quad \quad \quad \quad  
		  		  	+
		  		  	\big(1-f(c)\big) \,
		  		  	\bigg(
		  		  	\prod_{i \in I} (1-U) f(x_i)
		  		  	-
		  		  	\prod_{i \in I} f(x_i)
		  		  	\bigg)  		  	
		  		  	\Bigg]
		  		  	\\[7pt]
		  		  	\label{Gener_ModSLFV_Red}
		  		  	&
		  		  	\quad \quad \quad \quad \quad \quad \quad \quad 
		  		  	 \quad \quad  \quad \quad  \quad \quad \quad 
		  		  	\psi(x_1,\dots,x_n) \,
		  		  	dx_1 \dots dx_n \,
		  		  	\, dc.
				\end{align}
			   
				To understand this expression, we can think of $(X_t)_{t \geq 0}$ as a jump process
				with possibly an infinite number of jumps at each instant.
				The transitions of the process are indexed by the points $(t,c,v)$ of the Poisson Point process $\Pi$ with
				intensity $dt \otimes dc \otimes dv$. Morally, we can use equation (\ref{Single_Transition}) and write
				the generator in the form
	  		   \begin{align*}
		  		  	\nonumber
		  		  	\L I_n(f;\psi)
		  		  	=
		  		  	\int_{\R^d}
		  		  	\int_{[0,1]}
		  		  	\Big[I_n \big( f_{(c,v)};\psi \big) 
		  		  	- I_n \big( f;\psi \big)\Big]
		  		  	dv \,
		  		  	dc,
				\end{align*}
				where 
				$$ f_{(c,v)}(x)=f(x)+U \1_{\{\|x-c\| \leq R\} }\left(\1_{\{v \leq f(c)\} }-f(x)\right).$$
				When we replace the test functions $I_n$ with their expression, we obtain
	  		   \begin{align*}
		  		  	\nonumber
		  		  	\L I_n(f;\psi)
		  		  	=
		  		  	&
		  		  	\int_{\R^d}
		  		  	\int_{(\R^d)^n}
		  		  	\int_{[0,1]}
		  		  	\Big[
		  		  	\prod_{i=1}^n\, f_{(c,v)}(x_i)
		  		  	- \prod_{i=1}^n\, f(x_i)
		  		  	\Big] dv \,
		  		  	\\[7pt]
		  		  	&
		  		  	\quad \quad \quad \quad \quad \quad \quad \quad 
		  		  	\psi(x_1,\dots,x_n) \,
		  		  	dx_1 \dots dx_n \,
		  		  	dc.
				\end{align*}
				To express the integral 
				\begin{align*}
					& \ 
					\int_{[0,1]} 
		  		  	\Big[
		  		  	\prod_{i=1}^n\, f_{(c,v)}(x_i)
		  		  	- \prod_{i=1}^n\, f(x_i)
		  		  	\Big] dv,
				\end{align*}
				we find the unique set $I \subset \{1,\dots,n\}$ that verifies $x_j \in B(c,R)$ if and only if $j \in I$ (see figure \ref{Fig_Generator} ).
				After careful computations, we obtain expression (\ref{Gener_ModSLFV_Red}).
					  		
				\begin{figure}
					\includegraphics{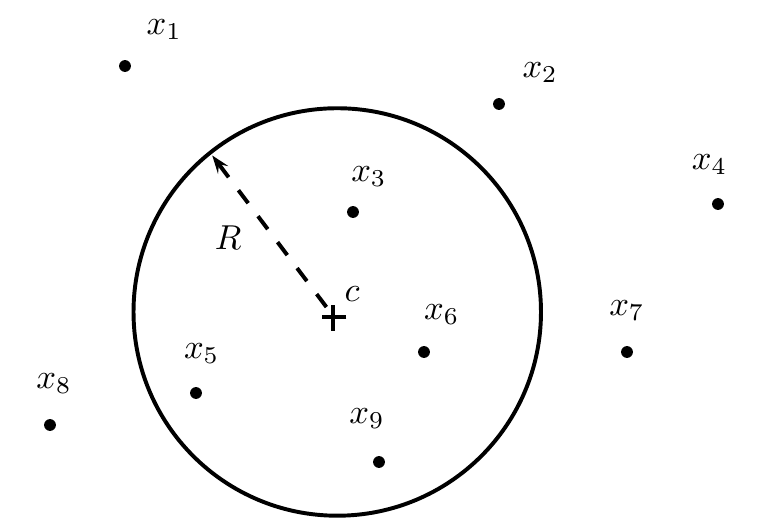}
					\caption{Visualisation of the unique set $I$ that produces a nonzero term in the sum appearing in the definition (\ref{Gener_ModSLFV_Red}) of the generator. In this figure, $d=2$, $n=9$, and the set $I=\{3,5,6,9\}$ indexes the points $x_i$ that belong to the ball of centre $c$ and radius $R$.}
					\label{Fig_Generator}
				\end{figure}

				One needs to prove that there exists a Markov process $(X_t)_{t\geq0}$
		  		that is defined by (\ref{Gener_ModSLFV_Red}). 
		  		A general proof for the existence of the S$\Lambda$FV process is given in \cite{BARTON_ETHERIDGE_VEBER_2010} using duality. However, we do not need to use
		  		this result, 
		  		because in our case we are able to construct directly the process
		  		using our forward in time method, see \S \ref{CTMC}.								
				
		\subsection{Main result}
											
					The bounded support $Red$ population is
					competing against the unbounded support $Blue$ population, so intuitively, we expect the $Red$ population
					to become extinct. The real question is how far the $Red$ population manages
					to spread before ultimately disappearing. This is why we are studying the dynamics of the support of the $Red$ population.
					
					Every time the individual sampled to be the parent is 
					not $Red$, the proportion of the $Red$ population decreases, which decreases the overall probability that the parent 
					at the next event is going to be $Red$. The same reasoning applies to the $Blue$ population. 
					The only way for the support to grow is if the ball of centre $c$ and radius $R$ is not entirely contained in the support of $X_t$,
					while the parent sampled at $c$ is $Red$. On the other hand, if we take $U<1$, the support never shrinks, as once 
					a point $y \in \R^d$ is occupied with a frequency $f(y)$, its future frequencies will always be positive.
					
					Given this schematic view of the dynamics of the process, one intuitively expects a behaviour similar 
					to what we are going to call the \emph{oil film spreading}:
					The proportion of the $Red$ population would converge to zero at every point, but at the same time
					its support would grow
					forever					
					and ultimately would occupy an infinite subset of $R^d$. Stated more naturally, there seems to be no reason
					why the support would not grow to an infinitely large set.\\
					
					However, the actual behaviour of the process is rather counterintuitive. Before stating our main result, 
					we need to introduce some notation.\\
					\begin{Notation}
						\begin{itemize}
							\item
						        For any given function
						        $f: \mathbb{R}^d \rightarrow \mathbb{R}$, we denote its support
			 							by 
						        Supp$(f)$.
							\item
								We define $\mathcal{S}_c$ to be the set of Borel measurable
								functions $f: \, 
						        \mathbb{R}^d \rightarrow [0,1]$ with compact support. We 
						        endow $\mathcal{S}_c$ with the $L^\infty$ norm.
							\item
								Given a set $A \subset \R^d$ and $R>0$, we denote by $A^R$ the $R$-expansion of $A$, 
								that is the set defined by
								$$ A^R:=\{x \in \R^d \text{ s.t. } \min_{y \in A} \|x-y \|\leq R\}.$$
				      \end{itemize}
      				\end{Notation}

		        	\begin{Theorem}
	        		\label{Main_Result}
			        	Let $(X_t)_{t \geq 0}$ be a Markov process
			        	with generator (\ref{Gener_ModSLFV_Red}). Suppose $X_0=f$ is deterministic
			        	and with bounded support.
						Then, there exists a random finite set $B \subset \mathbb{R}^d$, and 
						an  almost surely finite random 
						time $T$ such that
						\begin{equation}
							\forall t > T, \,
							\text{Supp}(X_t) = B \quad \text{a.s.}
						\end{equation}
						Furthermore, 
						\begin{equation}
							\sup_{z \in B} X_t(z) \rightarrow 0 \quad \text{a.s. as } t 
							\rightarrow 
							\infty	\\
						\end{equation}					
					\end{Theorem}	

			        \begin{Remark}
			        	For the sake of clarity, we chose $X_0$ to be deterministic. The result would still be true if 
			        	$\E[|\text{Supp}(X_0)|]<\infty$, see Remark \ref{Rmk_Delta0_Random}.
			        \end{Remark}
			        
			        The proof of this result is the objective of this paper. It is fairly challenging because of the large dimension of the
			        state space. Although the result is stated for the support of $X_t$, the actual object we need to keep track of
			        is the whole function $X_t$. The natural approach consisting in approximating probabilities of trajectories that correspond to the event
			        in question simply do not work, because such approximations waste too much information about the process.
			        Our approach is to first summarise the structure
			        of the process in a useful way. This is why we build adequate geometric tools that allow us to use a powerful martingale argument.\\
					
					Fundamentally, the cause of the behaviour described in Theorem \ref{Main_Result} lies in the discrete nature of the jumps inherent to $\Lambda$FV processes, rather than in the
					geometry of the S$\Lambda$FV process. 
					
					To see that, we consider the simpler process $(\hat{Z}_t)_{t \geq 0}$ where there is no space, that is to say at every reproduction event, 
					the parent is sampled with probability $\hat{Z}_t$, and
					a proportion $U$ 
					of the whole population is replaced by the offspring of the parent.
					In the notation of \cite{BERTOIN_LEGALL_2003},
					the process $(\hat{Z}_t)_{t \geq 0}$ is the $\Lambda$FV process 
					on the state-space $K=\{Red,Blue\}$ with
					$\Lambda$-measure $\Lambda(du)=u^2 \delta_U (du)$ .
					It is a continuous time Markov Chain on $[0,1]$ with constant intensity, and the transitions of its embedded
					discrete time Markov Chain $(Z_n)_{n \geq 0}$ are given by
					\begin{align*}
						\left\{
						\begin{aligned}
							&
							Z_{n+1}=(1-U)Z_n+U \, \varepsilon_{n+1},
							\\
							&
							\varepsilon_{n+1} \, | \, Z_n \, \thicksim \, \text{B}(Z_n),
						\end{aligned}
						\right.
					\end{align*}
					where $\text{B}(Z_n)$ is a Bernoulli distribution with parameter $Z_n$.
					It is straightforward to show that $(Z_n)_{n \geq 0}$ is a nonnegative martingale, 
					and 
					therefore converges almost surely. As a consequence, $\varepsilon_n$ 
					converges 
					almost
					surely to $0$ or $1$.
					This means that after some
					finite random time, $\varepsilon_n$ will remain constant equal to $0$
					or $1$. Almost surely in finite time, 
					either
					the $Red$ or the $Blue$ population will be the only one to keep 
					reproducing. 
					
					This remarkable feature is due to the fact that if the frequency $Z_n$ is 
					not sampled 
					a few
					times in a row, it is going to decrease geometrically, and becomes rapidly too small 
					to be 
					sampled again. The same reasoning applies to the $Blue$ population.\\
					
					As we will prove, the same mechanism takes place in the spatial model, that is
					after some almost surely finite random time, the $Red$ population will stop reproducing.
					
		\subsection{Proofs and outline}

					The martingale argument we demonstrated in the previous section seems to be the most promising approach.
					However, to be able to use such an argument, we need to find a way to filter out all the complex dependencies introduced by space,
					which is the main challenge in this work. We solved this problem by introducing in \S \ref{Forbid_Reg} the geometrical object 
					of \emph{forbidden region} that 
					 allows	to connect the martingale convergence to the sampling of the $Red$ population.
					
					The rest of this article
					is devoted to the proof
					of Theorem \ref{Main_Result}, as well as the
					construction of the process $(X_{(t)})_{t\geq0}$ defined by (\ref{Gener_ModSLFV_Red}).
					
					In Section \ref{Sec_DTMC}, we introduce a discrete time Markov Chain $(Y_n)_{n \geq 0}$ which is the discrete time
					equivalent of $(X_{(t)})_{t\geq0}$.
					This chain is going to be used to construct $(X_{(t)})_{t\geq0}$ and to prove 
					Theorem \ref{Main_Result}.
					 We state in Proposition \ref{Conv_Delta_n} the equivalent of Theorem \ref{Main_Result} in discrete time.

					 Section \ref{Sec_Geometry} provides a toolbox that allows to handle easily the geometry of the model.
					 
					 We prove in Section \ref{Finit_Sampl}
					 the central Proposition \ref{Prop_Finit_Many_Sampl}, which states that in discrete time, the $Red$		 					 population defined by $Y_n(\cdot)$ is sampled 
					 from only finitely many times. The proof relies on the fact that the total mass of the population, i.e. the integral of $Y_n(\cdot)$
					 over $\R_d$, is a martingale that converges almost surely. In \S \ref{Forbid_Reg}, we introduce the crucial concept of forbidden region, and use it to 
					 prove Proposition \ref{Prop_Finit_Many_Sampl}.
					 
					 We gather all the results in Section \ref{Prf_Thrm}	. Proposition \ref{Prop_Finit_Many_Sampl}
					 allows both to use $(Y_n)_{n\geq0}$ to construct $(X_t)_{t\geq0}$ as a non explosive continuous time Markov Chain, and to prove
					 Theorem \ref{Main_Result}.
					 
					 We finally conclude by discussing some extensions of this work.

	\section{A discrete time Markov chain}
	\label{Sec_DTMC}
		
		\subsection{Construction}

				\begin{Definition}
					\label{Def_Yn}
					Consider $R>0$ and $0<U<1$. Let $y$ be an 
					$\mathcal{S}_c$-valued random 
					variable.
					We construct simultaneously random sequences $(C_n)_{n \geq 
					1}$
					and $(V_n)_{n \geq 1}$, a filtration $(\mathcal{P}_n)_{n \geq 
					0}$, 
					and an $\mathcal{S}_c$-valued Markov chain $(Y_n)_{n\geq0}$,
					using the following recurrence:
					\begin{align}
						\left\{
						\begin{aligned}
							&
							Y_0=y,
							\\
							&
							\mathcal{P}_0:=\sigma(Y_0),
							\\
							&
							\Delta_0=\text{Supp}(Y_0).
						\end{aligned}
						\right.
					\end{align}
					and for $n\geq0$,
					\begin{itemize}
						\item $\mathcal{P}_n:=
							\sigma(C_1,\dots,C_n,V_1,\dots,V_n,
							Y_0,\dots,Y_n)$,
						\item conditionally on $\mathcal{P}_n$, $C_{n+1}$ is uniform 
							on $\Delta_n^{R}$,
						\item $V_{n+1}$ is distributed uniformly on $[0,1]$, 
						  independently from 
							$\mathcal{P}_n$ and $C_{n+1}$, 
						\item $Y_{n+1}$ is given by the formula
							\begin{align}
				  		\label{Dyn_Yn}							
								Y_{n+1}(\cdot)=Y_n(\cdot)+U \, \delta_{B(C_{n+1},R)}(\cdot) \, 
								\bigl( \mathds{1}_{ \{ 
				  			V_{n+1} \leq 	Y_n(C_{n+1})\} } -Y_n(\cdot)	\bigr)
				  		\end{align}
				  	\item $\Delta_{n+1}=\text{Supp}(Y_{n+1})$.				
		  		\end{itemize}
				  We introduce the notation $\varepsilon_{n+1}:=\mathds{1}_{ \{ 
				  V_{n+1} \leq 	 Y_n(C_{n+1})\} }$.
				  In 
				  particular, the trajectories of $(\Delta_n)_{n \geq 0}$ are 
				  given by
	        \begin{equation}
	        	\Delta_{n}=\Delta_{0} \cup  
	        	\bigcup_{
	        		\substack{
	        			1\leq k \leq n,\\
	        			\varepsilon_{k}=1	}
	        		}
	        	 B(C_{k},R).
	        	 \label{Dyn_Clust}
	        \end{equation}
	        Finally, we denote the natural filtration of $(Y_n)_{n\geq0}$ 
	        by
					$\mathcal{G}_n:=\sigma(Y_0,\dots,Y_n)$.
				\end{Definition}				
				
				\begin{Remark}
					We recall that $\Delta_n^R$ is the $R$-expansion of the set 
					$\Delta_n$.
				\end{Remark}
				
			  At each 
			  reproduction event, the random variable $C_{n+1}$
			  corresponds to the centre of the event. The parent is sampled uniformly
			  at location $C_{n+1}$ thanks to the random variable $V_{n+1}$, so that the parent is $Red$ with probability
			  $Y_n(C_{n+1})$, and $Blue$ with probability $1-Y_n(C_{n+1})$.
			  The constants $R$ and $U$ are the radius of the 
			  event and the proportion of the population that is modified. 
			  The 
			  random variable $\varepsilon_{n+1}$ indicates the types of the 
			  parent chosen. 
			  
			  \begin{Notation}
				  If
				  $\varepsilon_{n+1} =1$, we say that the event is a 
				  \emph{positive 
				  sampling event}, because
				  the total red population increases, whereas when 
				  $\varepsilon_{n+1} 
				  =0$ we say that it 
				  is a \emph{negative sampling event}.
				\end{Notation}
			  Equation (\ref{Dyn_Clust}) shows that if the cluster $\Delta_n$ 
			  is to increase, the 
			  minimum requirement is that there is a positive sampling.\\

        \begin{Remark}
        	\label{Delta_Expands}
        	Expression (\ref{Dyn_Clust}) is true because we assumed
        	$U<1$. In this case, the support and the range of the process 
					coincide. Once a region 
					is occupied by the \emph{Red} population it remains occupied 
					at every finite time. Therefore the cluster $\Delta_n$ never 
					shrinks. In the case where $U=1$ expression (\ref{Dyn_Clust}) 
					would remain true if $\Delta_n$ was defined to be the range of the process,
					that is $\bigcup_{n \geq 0} \text{Supp}(Y_n)$.
        \end{Remark}
		
		\subsection{Result of the cluster convergence}
				
				The following result is the expression of our main result
				in the discrete time setting, with the temporary technical condition 
				$Y_0=a\, \delta_{B(C_0,r_0)}$. This condition is removed in Proposition
				\ref{Conv_Delta_n_General} by allowing $Y_0$ to be any deterministic
				function with bounded support.
				
				\begin{Proposition}
					\label{Conv_Delta_n}
					Suppose $Y_0=a\, \delta_{B(C_0,r_0)}$, where 
					$a \in [0,1]$, $r_0>0$ and 
 					$C_0 \in \mathbb{R}^d$.
					Then, there exists an almost surely finite 
					random time $\kappa$ such that
					\begin{align}
						\forall n > \kappa, \quad \varepsilon_n=0.
					\end{align}
					Therefore, there exists an
					almost surely 
					bounded random set $B \in \mathbb{R}^d$ such that
					\begin{align}
						\forall n > \kappa, \quad \Delta_n=B.
					\end{align}
				\end{Proposition}
				Most of the remainder of this paper is devoted to proving this 
				Proposition. We first 
				investigate the geometric properties of the model in the next section.

	\section{Geometry}
	\label{Sec_Geometry}

				This section constructs all the tools that allow us to manage the geometry of the process.
				
        \begin{Remark}
        	From now on, unless specified otherwise, we suppose that $Y_0=a\, \delta_{B(C_0,r_0)}$.
        \end{Remark}
        
		\subsection{$Y_n$ is piecewise constant}
		
				\begin{Definition}
					For notational convenience, we introduce the sequence\\ 
					$(R_n)_{n \geq 0}$ such that
					\begin{align}
						\left\{
						\begin{aligned}
							& R_0=r_0,
							 &
							\\
							& R_n=R,
							 & n \geq 1.
						\end{aligned}
						\right.
					\end{align}
				\end{Definition}
				
 				\begin{Lemma}
 				\label{Struct_Yn}
 					For every $n \geq 0$, for every 
 					$\zeta \subset \{0,\dots,n\}$,
 					consider the set $A_{n,\zeta}$ defined by
 					\begin{align}
 					\label{Struct_A}
 						A_{n,\zeta}:=
				  	\big( \bigcap_{m \in \zeta} B(C_{m},R_{m}) \big)
				  	 \, \diagdown \, \big( \bigcup_{\substack{j \leq n,\\ j \notin \zeta}}
				  	B(C_{j},R_{j}) \big).
 					\end{align}
 					The function $Y_n$ can be written as
				  \begin{align}
				  	\label{Spat_Struct}
				  	Y_n=\sum_{\zeta \subset \{0,\dots,n\}}
				  	\alpha_{n,\zeta} \, \delta_{A_{n,\zeta}},
				  \end{align}
 					where the sets $A_{n,\zeta}$ are all disjoint for a given $n$,
 					and $\alpha_{n,\zeta} \geq 0$.
				\end{Lemma}
				
				\begin{Remark}
					By construction, $\forall z \in A_{n,\zeta}$, we have $Y_n(z)=\alpha_{n,\zeta} $.
				\end{Remark}
				
		  		
				\begin{figure}
					\includegraphics{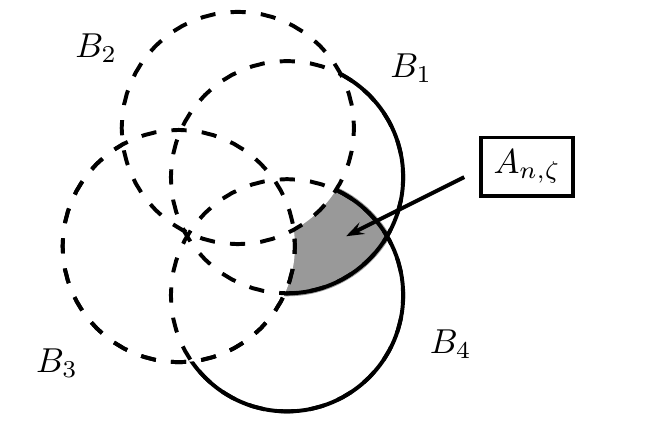}
					\caption{Structure of the level sets $A_{n,\zeta}$ defined in ($\ref{Struct_A}$), with $n=4$ and $\zeta:=\{1,4\}$. We use the notation $B_j:=B(C_j,R_j)$.}
					\label{Fig_Structure_Levels_Y}
				\end{figure}

				\begin{proof}
					We first introduce the shorter notation
					$$B_j:=B(C_{j},R_{j}), \ j \geq 0.$$
					The fact that the sets $A_{n,\zeta}$ are all disjoint for a given $n$
					is straightforward, so we just need to prove (\ref{Spat_Struct}).
					But before that, we need to show that
					\begin{align}
					\label{Union}
						\bigcup_{\zeta \subset \{0,\dots,n\}} A_{n,\zeta}
						=
						\bigcup_{i=1}^{n} B_i,
					\end{align}
					and we proceed by induction on $n$.
					The statement is true for $n=0$.
					Suppose now that it is true for some given $n \geq 0$. 
					Let $\zeta' \subset \{0,\dots,n+1\}$.
					\begin{itemize}
						\item[--] If $\zeta'=\{n+1\}$, then $A_{n+1,\zeta'}=B_{n+1}
							 \, \diagdown \, \bigcup_{\zeta \subset \{0,\dots,n\}} A_{n,\zeta}.$
	
						\item[--] If $(n+1) \in \zeta'$, and $\zeta' \neq \{n+1\}$
							then there exists $\zeta \subset \{0,\dots,n\}$
							such that 
							$$A_{n+1,\zeta'}=B_{n+1} \cap A_{n,\zeta}.$$
						
						\item[--] If $(n+1) \notin \zeta'$,
							then there exists $\zeta \subset \{0,\dots,n\}$
							such that 
									$$A_{n+1,\zeta'}=A_{n,\zeta} \, \diagdown \, B_{n+1}.$$
					\end{itemize}
					Therefore, we see that 
					\begin{align*}
						\bigcup_{\zeta' \subset \{0,\dots,n+1\}} A_{n+1,\zeta'}
						= 
						\, &
						\Big(
						\bigcup_{\zeta \subset \{0,\dots,n\}} B_{n+1} \cap A_{n,\zeta}
						\Big) \,
						\cup \,
						\Big(
						\bigcup_{\zeta \subset \{0,\dots,n\}} A_{n,\zeta} \, \diagdown \, B_{n+1}
						\Big)
						\\[7pt]
						&
						\quad \quad \quad \quad \quad \quad \quad \quad 
						\quad \quad \quad
						\cup
						\Big(
						 B_{n+1} \, \diagdown \, 
						\bigcup_{\zeta \subset \{0,\dots,n\}} A_{n,\zeta}
						\Big)
						\\[7pt]
						=
						& \,
						\Big(
						\bigcup_{\zeta \subset \{0,\dots,n\}} A_{n,\zeta}
						\Big)
						\cup
						\Big(
						 B_{n+1} \, \diagdown \, 
						\bigcup_{\zeta \subset \{0,\dots,n\}} A_{n,\zeta}
						\Big)
						\\[8pt]
						=
						& \,
						B_{n+1} 
						\cup
						\Big(
						\bigcup_{\zeta \subset \{0,\dots,n\}} A_{n,\zeta}
						\Big),
					\end{align*}
					and the statement is proven using the inductive hypothesis.\\
					
					We can return to the proof of expression (\ref{Spat_Struct}). We need to show 
					that
					for all $n \geq 0$, $\zeta \subset \{0,\dots,n\}$,
					\begin{align}
					\label{Union_An}
						&
						\Big( 
						x \notin \bigcup_{\zeta \subset \{0,\dots,n\}} A_{n,\zeta} 
						\Big)
						\text{ implies }
						\Big(
						Y_n(x)=0
						\Big),
						\\
						\nonumber
						\text{and }
						&
						\\
						\label{xy_in_An}
						&
						\Big( x,y \, \in A_{n,\zeta} \Big)
						\text{ implies }
						\Big(
						Y_n(x)=Y_n(y)
						\Big).
						\\
						\nonumber
					\end{align}
					
					We prove (\ref{Union_An}) by induction on $n$, and we use (\ref{Union}).
					Statement (\ref{Union_An}) is satisfied for $n=0$  because
					$Y_0=a\, \delta_{B(C_0,r_0)}$.
					Suppose now that it is true 
					for some $n \geq 0$, and consider $x \notin \cup_{i=1}^{n+1} B_i$. In
					particular, 
					$x \notin B_{n+1}$, and using the dynamics equation (\ref{Dyn_Yn}),
					we find that
						$Y_{n+1}(x)=Y_n(x).$
					Because $x \notin \cup_{i=1}^{n} B_i$, we can use the inductive hypothesis, 
					and we obtain that $Y_{n}(x)=0$, which proves (\ref{Union_An}).
					
					To prove (\ref{xy_in_An}), we also use induction. It is true for $n=0$.
					Suppose (\ref{xy_in_An}) is satisfied
					for a given $n \geq 0$. Consider
					$\zeta' \subset \{0,\dots,n+1\}$, and take $x,y \in A_{n+1,\zeta'}$.
					\begin{itemize}
						\item[--] If $\zeta'=\{n+1\}$, then 
							$A_{n+1,\zeta'}=B_{n+1}
							 \, \diagdown \, \bigcup_{\zeta \subset \{0,\dots,n\}} A_{n,\zeta}$, and
							 in particular $$x,y \notin \bigcup_{\zeta \subset \{0,\dots,n\}} 
							 A_{n,\zeta}.$$
							We have $\delta_{B_{n+1}}(x)=\delta_{B_{n+1}}(y)=1$,
							and using (\ref{Union_An}) we see that
							$Y_n(x)=Y_n(y)=0$.

						\item[--] If $(n+1) \in \zeta'$ and $\zeta' \neq \{n+1\}$,
							there exists $\zeta \subset \{0,\dots,n\}$
							such that 
							$$A_{n+1,\zeta'}=B_{n+1} \cap A_{n,\zeta}.$$
							Therefore $\delta_{B_{n+1}}(x)=\delta_{B_{n+1}}(y)=1$,
							and using the inductive hypothesis, we have $Y_n(x)=Y_n(y)$.
						
						\item[--] If $(n+1) \notin \zeta'$,
							then there exists $\zeta \subset \{0,\dots,n\}$
							such that 
									$$A_{n+1,\zeta'}=A_{n,\zeta} \, \diagdown \, B_{n+1}.$$
							In this case $\delta_{B_{n+1}}(x)=\delta_{B_{n+1}}(y)=0$,
							and using the inductive hypothesis, we have $Y_n(x)=Y_n(y)$.
					\end{itemize}					
					We can express $Y_{n+1}$ using the dynamics equation (\ref{Dyn_Yn}),
					and we obtain:
					\begin{align*}
						\left\{
						\begin{aligned}
							&
							Y_{n+1}(x)=Y_n(x)+U \, \delta_{B_{n+1}}(x)
							\, (\varepsilon_{n+1}-Y_n(x))
							\\
							&
							Y_{n+1}(y)=Y_n(y)+U \, \delta_{B_{n+1}}(y)
							\, (\varepsilon_{n+1}-Y_n(y)).
						\end{aligned}
						\right.
					\end{align*}
					We have proved that for any choice of $\zeta'$, all the terms in the above 
					equations are the same for $x$ and $y$, therefore $Y_{n+1}(x)=Y_{n+1}(y)$,
					and we have proved (\ref{xy_in_An}).
				\end{proof}

		\subsection{Variation of the local average}

        A central tool for the rest of the work is
        the average of the function $Y_n$ on a ball of radius $R$ and 
        centre $x \in \mathbb{R}^d$. It is important that $R$ is the radius of the reproduction event,
        as this is what links the martingale introduced in the next section to the geometry of the process (see 
        Lemma \ref{Mass_Change}).
        We introduce the following function:
        
				\begin{Definition}
					\begin{align}
						\Phi_n(x):=\int_{B(x,R)} Y_n(z) \, dz.
					\end{align}
				\end{Definition}

				The main result of this section is the following.\\
			  \begin{Proposition}
			  	\label{Mean_Val}
			  	For every $x,y \in \mathbb{R}^d$,
			  	\begin{align}
			  		\Phi_n(y)-\Phi_n(x) \leq \|y-x\| S(R).
			  	\end{align}
			  \end{Proposition}
			
			$ $ \newline

			  The rest of this section is devoted to proving this inequality. 
			  For this we need to introduce some auxiliary functions.
			  
			  \begin{Definition}
			  	Given $x,y \in \mathbb{R}^d$, for every $n\geq0$, we define
			  	\begin{align}
			  	\label{Lambda_nxy}
			  		\Lambda_n^{x,y}: [0, \|y  -x\|] & \longrightarrow  
			  		[0,\infty)
			  		\\
			  		\nonumber
			  		 t
			  		 & 
			  		 \, \longmapsto 
			  		 \Lambda_n(t):=\Phi_n \Big( \,x+ t \,\frac{y-x}{\|y -x\|} \, 
			  		 \Big).
			  	\end{align}
			  \end{Definition}
			  The key property for the proof of Proposition \ref{Mean_Val} is 
			  the following.
			  \begin{Proposition}
			  	\label{Piece_Diff}
			  	$\Lambda_n^{x,y}$ is a continuous, piecewise differentiable 
			  	function.
			  	 Moreover, for every point $t$ where $\Lambda_n^{x,y}$ is 
			  	 differentiable,
			  	 we have:
			  	\begin{align}
			  	\label{Bound_D_Lambda}
			  		\frac{d \Lambda_n^{x,y}}{dt} (t) \leq S(R).
			  	\end{align}
			  \end{Proposition}
			  \begin{proof}
			  	We first prove that $\Lambda_n^{x,y}$ is continuous and that there is at most
			  	a finite number $J$ of points $t_1 < \dots < t_J$ at which $\Lambda_n^{x,y}$
			  	is not differentiable.
			  	Thanks to equation (\ref{Spat_Struct}) from Lemma 
			  	\ref{Struct_Yn},
			  	we see that $\Phi_n$ is given by
			  	\begin{align*}
			  		\Phi_n(x)&=\int_{B(x,R)} 
			  		\sum_{\zeta \subset \{0,\dots,n\}}
			  		\alpha_{n,\zeta} \delta_{A_{n,\zeta}}(z) \, dz
			  		\\
			  		&= 
			  		\sum_{\zeta \subset \{0,\dots,n\}}
			  		\alpha_{n,\zeta} |B(x,R) \cap A_{n,\zeta}|,
			  	\end{align*}
			  	therefore $\Lambda_n^{x,y}(t)$ is given by
			  	\begin{align}
 				\label{Struct_Lambda_n}
			  		\Lambda_n^{x,y}(t)&=
			  		\sum_{\zeta \subset \{0,\dots,n\}}
			  		\alpha_{n,\zeta} \, 
			  		|
			  		B\big( \,x+ t \,\frac{y-x}{\|y -x\|} \, 
			  		 ,R \big)
			  		 \cap A_{n,\zeta}|.
			  	\end{align}
			  	
			  	We simplify the notation by introducing 
			  	$B_j:=B(C_{j},R_{j}), \ j \geq 0$
			  	and
			  	$$x_t:=x+ t \, \frac{y-x}{\|y -x\|}.$$\\
			  	
			  	The definition (\ref{Struct_A}) of the sets $A_{n,\zeta}$ and the
			  	inclusion-exclusion formula allow us to prove that
			  	for each $\zeta \subset \{0,\dots,n\}$, there exists
			  	a $\beta_{n,\zeta} \in \R$ such that
			  	\begin{align*}
			  		\Lambda_n^{x,y}(t)=
			  		&
			  		\sum_{\zeta \subset \{0,\dots,n\}}
			  		\beta_{n,\zeta} \, 
			  		|
			  		B\big( \,x_t \, ,R \big)
			  		\cap 
			  		\big( \bigcap_{m \in \zeta} B_m \big)
			  		|.
			  	\end{align*}
			  	\begin{Remark}
					The main change with expression (\ref{Struct_Lambda_n}) is that now we are working with intersections of balls, which are convex,
					whereas the sets $A_{n,\zeta}$ are usually not. Also, we had the fact that $\alpha_{n,\zeta}$ is the value of the function $Y_n$
					on the set $A_{n,\zeta}$, and such an interpretation is lost for $\beta_{n,\zeta}$.
			  	\end{Remark}
			  	
			  	If we introduce the function $H_{\zeta}$ defined for each set
			  	$\zeta \subset \{0,\dots,n\}$ by
			  	\begin{align}
			  		\label{Function_H_zeta}
			  		H_{\zeta}:
			  		[0, \|y -x\|]
			  		&
			  		\longrightarrow
			  		\R
			  		\\
       		  t \
       		  &
       		  		\nonumber
		  		  \longmapsto
		  		  | \,
		  		  B(x_t,R) \cap 
		  		  \bigcap_{m \in \zeta} B_m
		  		  \, |,
			  	\end{align}
			  	then we can simply rewrite the function $\Lambda_n^{x,y}$ as 
			  	\begin{align}
 				\label{Struct_Lambda_n_Convex}
			  		\Lambda_n^{x,y}(t)=
			  		&
			  		\sum_{\zeta \subset \{0,\dots,n\}}
			  		\beta_{n,\zeta} \, 
			  		H_{\zeta}(t).
			  	\end{align}

		  		
				\begin{figure}
					\includegraphics{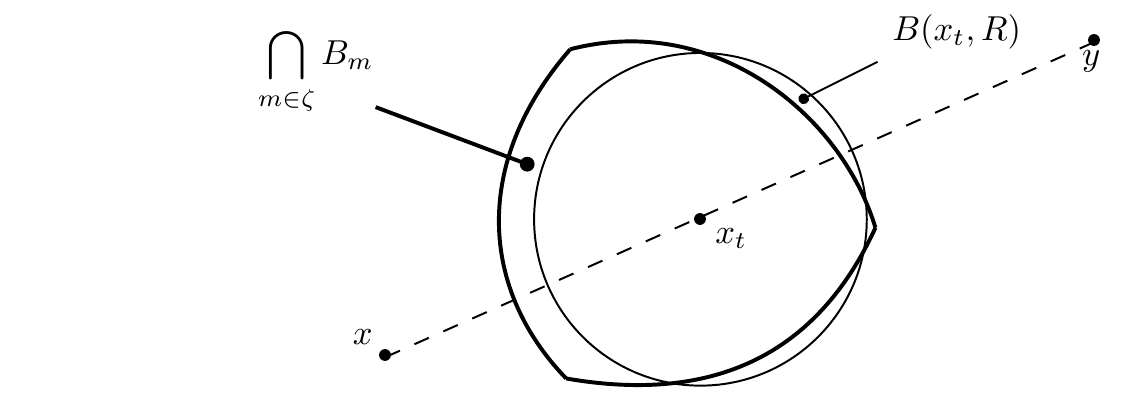}
					\caption{Visualisation of the function $H_{\zeta}$ defined in (\ref{Function_H_zeta}). The points $x$ and $y$ are fixed, as well as the balls 
					$B_m$, $m \in \zeta$, with $\zeta = \{1,2,3\}$ in this example.
					As the parameter $t$ varies between $0$ and $\|y -x\|$, the mobile point $x_t$ moves along the segment joining $x$ to $y$, with $x_0=x$ and 
					$x_{\|y -x\|}=y$. For a given $t$, $H_{\zeta}$(t) measures the volume of the intersection between the mobile ball $B(x_t,R)$ and the fixed set 
					$\bigcap_{m \in \zeta} B_m$. The function $\Lambda_{n}^{x,y}$ defined in (\ref{Lambda_nxy}) is expressed as a finite linear combination of such functions $H_{\zeta}$, see (\ref{Struct_Lambda_n_Convex})}
					\label{Fig_Function_Lambda_n_x_y}
				\end{figure}

			  	The continuity of $H_{\zeta}$ follows from the continuity of 
			  	the function $t \mapsto x_t$, and this shows that 
			  	$\Lambda_n^{x,y}$ is continuous.\\

			  	The set $\bigcap_{m \in \zeta} B_m$ is convex,
			  	therefore there exist $t_1$, $t_2$ 
			  	such that 
			  	\begin{align}
			  	\label{t1t2}
			  		H_{\zeta}(t)>0  \Leftrightarrow t_1<t<t_2.
			  	\end{align}
			  	Consider $t$ such that $t_1<t<t_2$. Thanks to (\ref{t1t2}), 
			  	this means we can choose a point $z$ belonging to 
			  	the 
			  	interior of
			  	$B(x_t,R) \cap \bigcap_{m \in \zeta} B_m$. 
			  	Because the set $B(x_t,R) \cap \bigcap_{m \in \zeta} B_m$ is convex, we can express its volume
			  	in $d$-dimensional spherical coordinates with $z$ as the new origin. Given angular coordinates
			  	$\phi:=(\phi_1,\dots,\phi_{d-1})$, we denote by $p_{\phi}(t)$ the unique point
			  	of the boundary of $B(x_t,R) \cap \bigcap_{m \in \zeta} B_m$ 
			  	with angular coordinates $\phi$, and $L(t,\phi)$ 
			  	the distance between $z$ and $p_{\phi}(t)$. We have:
			  	\begin{align*}
			  		H_{\zeta}(t)=
			  		\int_0^{\pi}
			  		\dots
			  		\int_0^{\pi}
			  		\int_0^{2 \pi}
			  		\int_0^{L(t,\phi)}
			  		r^{d-1}
			  		\sin^{d-2}(\phi_1)
			  		\sin^{d-3}(\phi_2)
			  		\dots
			  		\sin(\phi_{d-2})
			  		\\
			  		dr \, 
			  		d\phi_{d-1} \dots  d\phi_1.
			  	\end{align*}

				\begin{figure}
					\includegraphics{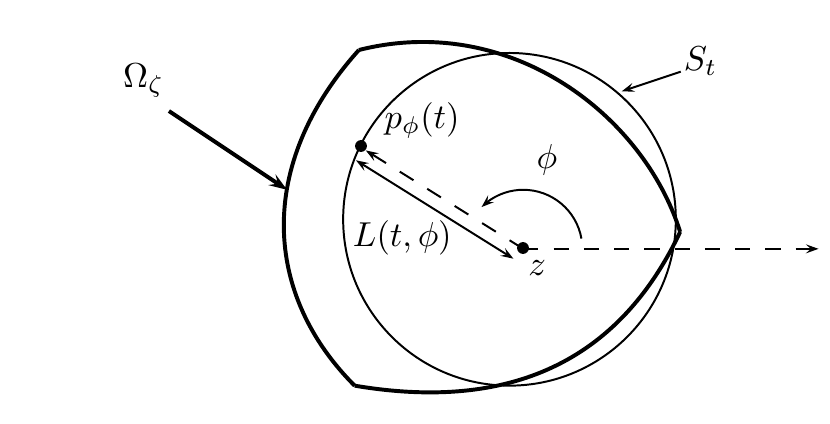}
					\caption{Illustration of $p_{\phi}(t)$ in dimension 
					2.
					The point $z$ is chosen arbitrarily inside the intersection of $\bigcap_{m \in \zeta} B_m$ and $B(x_t,R)$, and is taken to be
					the new origin. The boundaries of $\bigcap_{m \in \zeta} B_m$ and $B(x_t,R)$ are denoted respectively by $\Omega_{\zeta}$ and
					$S_t$.
					The angle 
					$\phi$ is defined in the local polar coordinate system, and for a given 
					angle $\phi$
					the point $p_{\phi}(t)$ is the projection of $z$ along the angle $\phi$ on the boudary of $\bigcap_{m \in \zeta} B_m \cap B(x_t,R)$.
					For some
					values of $\phi$ it belongs to the mobile sphere $S_t$, and for other values 
					it belongs to the static surface $\Omega_{\zeta}$.
					}
					\label{Fig_Convex_Disc}
				\end{figure}

			  	We denote by $\Omega_{\zeta}$ the boundary of $\bigcap_{m \in \zeta} B_m$, and by 
			  	$S_t$ the sphere of centre $x_t$ and radius $R$. We can find a partition 
			  	$\Xi_1,\dots,\Xi_K$ of the space
			  	$$ [0,\pi]^{d-2} \times [0,2 \pi]$$
			  	such that
			  	\begin{align*}
			  		\left\{
			  		\begin{aligned}
			  			&
			  			\forall \phi \in \Xi_j, \, p_{\phi}(t) \in \Omega_{\zeta},
			  			\\
			  			&
			  			\text{or}
			  			\\
			  			&
			  			\forall \phi \in \Xi_j, \, p_{\phi}(t) \in S_t.
			  		\end{aligned}
			  		\right.
			  	\end{align*}
			  	Therefore we can write $H_{\zeta}(t)$ as
			  	\begin{align*}
			  		H_{\zeta}(t)=
			  		\sum_{j=1}^K
			  		\int_{\Xi_j}
			  		\int_0^{L(t,\phi)}
			  		r^{d-1}
			  		\sin^{d-2}(\phi_1)
			  		\sin^{d-3}(\phi_2)
			  		\dots
			  		\sin(\phi_{d-2})
			  		dr \, d\phi.
			  	\end{align*}

			  	Thanks to this representation, we see that a sufficient condition 
			  	for	$H_{\zeta}$
			  	to be differentiable at $t$, $t_1<t<t_2$,  is that 
			  	for every $\phi$ in the interior of every $\Xi_j$, the function
			  	$t \mapsto L(t,\phi)$ is differentiable at $t$. In this case, 
			  	the derivative is given by
			  	\begin{align*}
			  		\frac{d H_{\zeta}(t)}{dt}=
			  		\sum_{j=1}^K
			  		\int_{\Xi_j}
						\frac{\partial L}{\partial t}
						(t,\phi)
						\,
			  		L(t,\phi)^{d-1}
			  		\sin^{d-2}(\phi_1)
			  		\sin^{d-3}(\phi_2)
			  		\dots
			  		\sin(\phi_{d-2})
			  		\, d\phi.
			  	\end{align*}			  	
			  	
			  	We focus now on the differentiability of $t \mapsto L(t,\phi)$, where
			  	$\phi$ belongs to the interior of $\Xi_j$ for some $j$.
			  	
			  	Suppose first that 
			  	$p_{\phi}(t) \in \Omega_{\zeta}$. Because the function 
			  	$t \mapsto x_t$ is continuous, 
			  	and because $z$ is a fixed point,
			  	there exists $h>0$ such that
			  	for every $u \in (t-h,t+h)$, $p_{\phi}(u) \in \Omega_{\zeta}$. 
			  	Therefore, there exists $h>0$ such that
			  	for every $u \in (t-h,t+h)$, $L(u,\phi)=L(t,\phi)$, and $t \mapsto 
			  	L(t,\phi)$
			  	is differentiable at $t$. 
			  	
			  	In the case where $p_{\phi}(t) \in S_t$, by the 
			  	same continuity argument, we obtain that for every $u \in (t-h,t+h)$,
			  	$p_{\phi}(u) \in S_{u}$. Because the distance between $z$ and the 
			  	projection of $z$ on	$S_t$ along the angle $\phi$ is differentiable,
			  	we conclude that $t \mapsto L(t,\phi)$ is also differentiable in 
			  	this case.\\
			  				  	
			  	We have proved that for all $t \neq t_1,t_2$, $H_{\zeta}$ is 
			  	differentiable at 
			  	$t$.
			  	Given that
			  	$$\Lambda_n^{x,y}(t)=\sum_{\zeta \subset \{0,\dots,n\}}
			  	\beta_{n,\zeta} \, H_{\zeta}(t),
			  	$$
			  	we conclude that there is at most a finite number of points at which 
			  	$\Lambda_n^{x,y}$ is not differentiable, and this proves the first 
			  	part of the 
			  	Proposition.\\
			  	
			  	We need now to show the upper bound 
			  	(\ref{Bound_D_Lambda}) for the derivative. Suppose
			  	$\Lambda_n^{x,y}$ is differentiable at $t$. By definition,
			  	\begin{align*}
			  		\Lambda_n^{x,y}(t)
			  		=
			  		\int_{B(x_{t},R)} Y_n(z) dz,
			  	\end{align*}
			  	therefore
			  	\begin{align*}
			  		\Lambda_n^{x,y}(t+h)-\Lambda_n^{x,y}(t)
			  		=
			  		\int_{B(x_{t+h},R)} Y_n(z) dz
			  		-
			  		\int_{B(x_{t},R)} Y_n(z) dz.
			  	\end{align*}			  	
			  	By construction, for all 
			  	$h \geq 0$, $B(x_{t+h},R) \subset B(x_t,R+h)$,
			  	which implies that
			  	\begin{align*}
			  		\Lambda_n^{x,y}(t+h)-\Lambda_n^{x,y}(t)
			  		& \leq
			  		\int_{B(x_{t},R+h)} Y_n(z) dz
			  		-
			  		\int_{B(x_{t},R)} Y_n(z) dz
			  		\\[7pt]
			  		& 
			  		\leq			  		
			  		\int_{B(x_{t},R+h) \diagdown B(x_{t},R)} Y_n(z) dz
			  		\\[7pt]
			  		& 
			  		\leq			  		
			  		|B(x_{t},R+h)|- |B(x_{t},R)|.
			  	\end{align*}			  	
			  	Dividing by $h$ and taking the limit as $h \rightarrow 0$, we obtain\\
			  	\begin{align}
			  		\frac{d \Lambda_n^{x,y}}{dt} (t) 
			  		\leq 
			  		\frac{d V(R)}{dR}
			  		=
			  		S(R),
			  	\end{align}
			  	where $V(R)$ is the volume of a ball of radius $R$, and $S(R)$
			  	its surface area.
			  \end{proof}
			  
			  \begin{proof}[Proof of Proposition \ref{Mean_Val}]
			  	We saw in the previous proof that there is only a finite 
			  	number of points 
			  	$t_1,\dots,t_J$ at which $\Lambda_n^{x,y}$ is not 
			  	differentiable.
			  	By continuity of $\Lambda_n^{x,y}$, and using 
			  	Proposition \ref{Piece_Diff},
			  	\begin{align*}
			  		&  \Phi_n(y)-\Phi_n(x)
			  		\\
			  		= & 
			  		\Lambda_n^{x,y}(\|y-x\|)-\Lambda_n^{x,y}(0)
			  		\\
			  		= & 
			  		\Lambda_n^{x,y}(\|y-x\|)-\Lambda_n^{x,y}(t_J) 
			  		+ \sum_{j=1}^{J-1} 
			  		\Lambda_n^{x,y}(t_{j+1})-\Lambda_n^{x,y}(t_j)
			  		\\
			  		& \quad \quad \quad \quad \quad \quad \quad \quad
			  		\quad \quad \quad \quad \quad \quad \quad 
			  		+\Lambda_n^{x,y}(t_1)-\Lambda_n^{x,y}(0)
			  		\\
			  		\leq & S(R) (\|y-x\|-t_J)
			  		+ \sum_{j=1}^{J-1} S(R)(t_{j+1}-t_j)
			  		\\
			  		& \quad \quad \quad \quad \quad \quad \quad \quad
			  		\quad \quad \quad \quad \quad \quad \quad
			  		+S(R)(t_1-0)
			  		\\
			  		\leq & \|y-x\| \, S(R).	
			  		\qedhere	  		
			  	\end{align*}
			  \end{proof}

	\section{Probability}
	\label{Sec_Finit_Sampl}
	
		\subsection{A martingale argument}
		\label{MartArg}
		
				\begin{Definition}
					We denote by $M_n$ the total mass of $Y_n$, that is
					\begin{equation*}
						M_{n}=\int_{\mathbb{R}^d} Y_{n}(z) dz.
					\end{equation*}
				\end{Definition}

				\begin{Lemma}
					\label{Mass_Change}
					The change of the total mass $M_{n+1}-M_n$
					is given by
					\begin{align}
						M_{n+1}-M_n 
						= U \, \bigl( \varepsilon_{n+1} V(R) 
						- \Phi_n(C_{n+1}) \bigr)
					\end{align}
				\end{Lemma}
				\begin{proof} Using (\ref{Dyn_Yn}),
					\begin{align*}
						M_{n+1}-M_n & 
						= \int_{\mathbb{R}^d} 
						\Big(
						Y_{n+1}(z)-Y_{n}(z)
						\Big)
						 dz
						\\
						& = \int_{\mathbb{R}^d} 
						U \, \delta_{B(C_{n+1},R)} (z) \, \bigl( \varepsilon_{n+1} 
						-Y_n(z)	
						\bigr)
						dz\\
						& = U \, \displaystyle \int_{B(C_{n+1},R)} 
						\Big(
						\varepsilon_{n+1} -Y_n(z)	
						\Big)
						dz\\
						& = U \, \bigl( \varepsilon_{n+1} V(R) 
						- \Phi_n(C_{n+1}) \bigr)
						\qedhere
					\end{align*}				
				\end{proof}

				\begin{Proposition}
					\label{Mart}
					$(M_n)_{n\geq0}$ is a discrete time nonnegative 
					$(\mathcal{G}_n)_{n \geq 0}$ martingale.
				\end{Proposition}
			
				\begin{proof}					
					
					Thanks to Lemma \ref{Mass_Change} and Definition \ref{Def_Yn},
					we can calculate explicitly 
					$\mathbb{E}[\,M_{n+1}-M_n \, | \, \mathcal{G}_n \,]$ as 
					follows:
					
					\begin{align*}
						 & \, \mathbb{E}\bigl[M_{n+1}-M_n \, | \, \mathcal{G}_n\bigr]
						\\[10pt]
						= & \,
						\mathbb{E}\biggl[U \, 
						\mathds{1}_{ \{ V_{n+1} \leq 	Y_n(C_{n+1})\} }
						 \int_{B(C_{n+1},R)}	dz 
							- U \, \int_{B(C_{n+1},R)} Y_n(z)	dz  \, | \, 
							\mathcal{G}_n\biggr]
						\\[10pt]
						= & U \,	
						\int_{\mathbb{R}^d}
						\int_{[0,1]}
						\biggl[
						\mathds{1}_{ \{ v \leq 	Y_n(c)\} }
						 \int_{B(c,R)}	dz 
							- \int_{B(c,R)} Y_n(z)	\,dz \,
						\biggr]\,
						dv \,
						\frac{\mathds{1}_{c \in \Delta_n^R}}{ |\Delta_n^R|}
						\,	dc
						\\[10pt]
						= & U \, 		
						\int_{\mathbb{R}^d}
						\biggl[
						Y_n(c)
						 \int_{B(c,R)}	dz 
							- \int_{B(c,R)} Y_n(z)	dz \,
						\biggr]\, \frac{\mathds{1}_{c \in \Delta_n^R}}{ |\Delta_n^R|}
	  				dc
	  				\\[10pt]
	  				= & \frac{U}{ |\Delta_n^R|} \, \biggl[ \,	
						\int_{\mathbb{R}^d}
						\int_{\mathbb{R}^d}
						\mathds{1}_{z \in B(c,R)} \,
						\mathds{1}_{c \in \Delta_n^R}
						\,
						Y_n(c) dz \, dc
						\\[10pt]
						& \quad \quad \quad \quad \quad \quad \quad \quad \quad 
						\quad \quad 
						\quad 
						-
						\int_{\mathbb{R}^d}
						\int_{\mathbb{R}^d}					
						\mathds{1}_{z \in B(c,R)}\,
						\mathds{1}_{c \in \Delta_n^R}
						\,
						Y_n(z)	dz \, dc \, \biggr].
					\end{align*}
					In particular, for every $f \in \mathcal{S}_c$,
					\begin{align*}
						& \int_{\mathbb{R}^d}
						\int_{\mathbb{R}^d}
	  				\mathds{1}_{\{z \in B(c,R)\}} \, 
	  				\mathds{1}_{\{c \in \text{Supp}(f)^R\}} \, 
	  				f(c)
	  				dz	\, dc \\
						= & \int_{\mathbb{R}^d}
						\int_{\mathbb{R}^d}
	  				\mathds{1}_{\{c \in B(z,R)\}} \, 
	  				f(c)
	  				dz	\, dc
	  				\quad \text{since } c \in B(z,R)
	  				\Leftrightarrow
	  				z \in B(c,R)	  				
	  				\\
						= & \int_{\mathbb{R}^d}
						\int_{\mathbb{R}^d}
	  				\mathds{1}_{\{z \in B(c,R)\}} \, f(z)
	  				dc	\, dz,\\
						= & \int_{\mathbb{R}^d}
						\int_{\mathbb{R}^d}
	  				\mathds{1}_{\{z \in B(c,R)\}} \, 
	  				\mathds{1}_{\{c \in \text{Supp}(f)^R\}} \, 
	  				f(z)
	  				dc	\, dz,
					\end{align*}					
				  so $\mathbb{E}[\, M_{n+1}-M_n \, | \, \mathcal{G}_n \, ]=0$, 
				  which shows that 
				  $(M_n)_{n\geq0}$ 
				  is a martingale.
			  \end{proof}

				\begin{Definition}
					Let $\alpha>0$ be a real number such that $0<\alpha<UV(R)/2$.					  We then define 
					\begin{equation}
						\tau_{\alpha}:=\inf \{p \geq 0: \forall n \geq p,\quad 
						|M_{n+1}-M_n| < \alpha 
						\}.
					\end{equation}
				\end{Definition}
				In particular, $\tau_{\alpha}$ is not a stopping time, but this 
				is not going to be an issue in what follows.

				\begin{Proposition}
					\label{Constraint_C}
					The random time $\tau_{\alpha}$ is a.s. finite, and
					$\forall n>\tau_{\alpha}$,
					\begin{align}
						\left\{
						\begin{aligned}
							& \Phi_n(C_{n+1}) < \frac{\alpha}{U}
							 & \text{if } \varepsilon_{n+1}=0,
							\\
							& \Phi_n(C_{n+1}) > V(R)-\frac{\alpha}{U}
							 & \text{if } \varepsilon_{n+1}=1.
						\end{aligned}
						\right.
					\end{align}
				\end{Proposition}

				\begin{proof}
					We know that $M_n$ is a nonnegative martingale, so 
					it converges almost surely when $n \rightarrow \infty$. 
					Therefore $\tau_{\alpha}$ is
					almost surely finite, and by definition,
					$\forall n>\tau_{\alpha}, \quad |M_{n+1}-M_n| < \alpha$.
					Using Lemma \ref{Mass_Change}, we observe that
					\begin{align*}
						|M_{n+1}-M_n| =
						\left\{
						\begin{aligned}
							& U \, \Phi_n(C_{n+1})
							 & \text{if } \varepsilon_{n+1}=0,
							\\
							& U \, \bigl( V(R) - \Phi_n(C_{n+1}) \bigr)
							 & \text{if } \varepsilon_{n+1}=1,
						\end{aligned}
						\right.
					\end{align*}
					which concludes the proof.
				\end{proof}				

			\subsection{Forbidden region}
				\label{Forbid_Reg}
                
				The concept of a forbidden region will allow us to treat 
				probabilistically the geometric 
				properties established in \S \ref{Sec_Geometry}.
				
				\begin{Lemma}
					For $n \geq 0$,
					\begin{align}
						\nonumber
						& \{n \geq \tau_{\alpha}\} \cap 
						\{\varepsilon_k=1 \text{ infinitely often} \}
						\\
						\label{Include_1}
						\subset
						\ &
						\{n \geq \tau_{\alpha}\} \cap \bigcap_{j=n}^{\infty}
						\{\sup \Phi_j > V(R)-\alpha /U\}
					\end{align}
				\end{Lemma}
				\begin{proof}
					Take $j \geq \tau_{\alpha}$ such that $\sup \Phi_j \leq 
					V(R)-\alpha /U$. 
					Using Proposition \ref{Constraint_C}, this implies that
					$\varepsilon_{j+1}=0$. In particular, 
					$\sup \Phi_{j+1} \leq \sup \Phi_j\leq V(R)-\alpha /U$. By 
					induction, we just
					showed that
					\begin{align*}
						\left\{
						\begin{aligned}
							& j \geq \tau_{\alpha}
							\\
							& \sup \Phi_j \leq V(R)-\alpha /U
						\end{aligned}
						\right.
						\Longrightarrow
						\left\{
						\begin{aligned}
							& j \geq \tau_{\alpha}
							\\
							& \forall k\geq j,	\varepsilon_k=0.
						\end{aligned}
						\right.
					\end{align*}
					The contrapositive of this implication allows to conclude this 
					proof.
				\end{proof}

				\begin{Definition}
					We define the \emph{forbidden region} $F_n$ to be
					\begin{align}
						& F_n:= \{ x \in \mathbb{R}^d \, :
						\frac{\alpha}{U} \leq \Phi_n(x) \leq
						V(R)-\frac{\alpha}{U}
						\}.
					\end{align}										
					We also introduce the quantity
					\begin{align}
						\psi:=V\Bigg(\frac{V(R)-2 \, \alpha/U}{S(R)}\Bigg).
					\end{align}
				\end{Definition}
				
				The reason for the name \emph{forbidden region} is motivated by 
				the following lemma, which tells us that after the time 
				$\tau_{\alpha}$, if the 
				local averages
				are always too high, then the points
				$C_{n+1}$ are forbidden from falling in the region $F_n$. Furthermore, 
				this lemma 
				provides a lower bound on the volume of $F_n$.
				\begin{Lemma}
					\begin{align}
					\nonumber
						& \{j \geq \tau_{\alpha}\} \cap
						\{\sup \Phi_j > V(R)-\alpha /U\}
						\\[10pt]
					\label{Include_2}
						\subset
						\ &
						\{C_{j+1} \notin F_j,|F_j|\geq \psi\}
					\end{align}
				\end{Lemma}
				\begin{proof}
					An immediate consequence of Proposition \ref{Constraint_C} is
					$$j \geq \tau_{\alpha} \Rightarrow C_{j+1} \notin F_j.$$
					More 
					work is required to obtain the lower bound for the volume of 
					the forbidden region. We first define
					$P_j:=\{ x \in \mathbb{R}^d \, :	\Phi_j(x) \geq
					V(R)-\alpha/U	\} $. If we assume 
					$\sup \Phi_j > V(R)-\alpha /U$, then $P_j$ is nonempty.	We can 
					then take a point $y \in P_j$.
								
					The function $\Phi_j$ is continuous, and $\Delta_j$ is finite, 
					so the region 
					$N_{j}:=\{ x \in \mathbb{R}^d \, :	\Phi_j(x) \leq 
					\alpha/U	\} $ is 
					infinite. Indeed, for every $x$ at a distance from $\Delta_j$ 
					larger than $R$,
					$\Phi_j(x)=0$. In particular, for a large enough positive number 
					$\overline{R}$, we can consider $\Gamma$ the sphere of 
					radius $\overline{R}$ and 
					centre $y$, and the ball $B(y,\overline{R})$ such that
					\begin{align}
					\label{Ball_Gamma_Delta}
						\left\{
						\begin{aligned}
							& \Delta_j \subset B(y,\overline{R}),
							\\
							& \Gamma \subset N_{j}. 
						\end{aligned}
						\right.
					\end{align}	
				
				For $x,y \in \R^d$, we denote by $[x,y]$ the line-segment 
				between $x$ and $y$.
				We need the following lemma:
				\begin{Lemma}
				\label{Two_Points}
					The point $y \in P_{j}$ being fixed, for every point $x \in 
					\Gamma$, we can 
					find 
					two points $x_0, y_0$ such that:
					\begin{align}
						\left\{
						\begin{aligned}
							& [x_0, y_0] \subset F_{j},
							\\
							& [x_0, y_0] \subset [x, y],
							\\
							& \|y_0-x_0\|\geq \frac{V(R)-2 \, \alpha/U}{S(R)}.
						\end{aligned}
					\right.
					\end{align}
				\end{Lemma}
				\begin{figure}
					\includegraphics{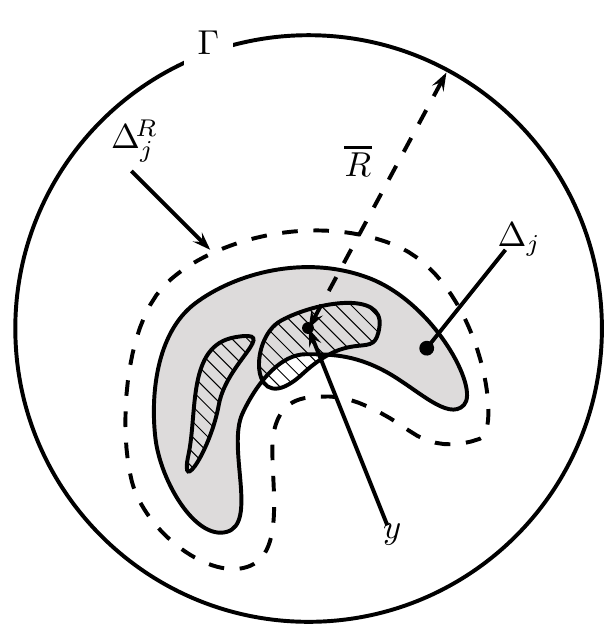}
					\caption{Illustration of the construction (\ref{Ball_Gamma_Delta}). The grey area $\Delta_j$ is the support of $Y_j$,
					and the hashed area is the region $P_j$. We choose arbitrarily $y \in P_j$.
					The dashed line is the boundary of $\Delta_j^{\!R}$, the set of points at distance at most $R$ from $\Delta_j$.
					For all $x \notin \Delta_j^{\!R}$, $\phi_j(x)=0$, therefore for large enough $\overline{R}$, the ball $B(y,\overline{R})$
					 and its boundary $\Gamma$ satisfy (\ref{Ball_Gamma_Delta}).}
					\label{Fig_Ball_Around_Delta}
				\end{figure}		  		
					By integrating the result of Lemma \ref{Two_Points} 
					over all the points $x \in \Gamma$, we find 
					that the volume of 
					$F_j$ is larger than the volume of a ball of radius
					$(V(R)-2 \, \alpha/U)/S(R)$, hence the 
					result.
				\end{proof}			
				\begin{proof}[Proof of Lemma \ref{Two_Points}]
					The function $\Lambda_j^{x,y}$ defined in (\ref{Lambda_nxy})
					is continuous, with 
					$\Lambda_j^{x,y}(0)=0$ and 
					$\Lambda_j^{x,y}(\|y-x\|)\geq V(R)-\alpha/U$, so 
					there
					are two points $t_1,t_2 \in [0,\|y-x\|]$ such that 
					$\Lambda_j^{x,y}([t_1,t_2])=
					[\alpha/U,V(R)-\alpha/U]$.
					By application of the continuous function 
					$t \longrightarrow x+ t \,(y-x)/(\|y -x\|)$, this means that 
					there are two points
					$x_0,y_0 \in \mathbb{R}^d$ such that 
					\begin{itemize}
						\item[--] $\Phi_j(x_0)=\alpha/U$,
						\item[--] $\Phi_j(y_0)=V(R)-\alpha/U$,
						\item[--] $\forall z \in (x_0,y_0), 
							\Phi_j(z)\in (\alpha/U,V(R)-\alpha/U)$.
					\end{itemize}
					The last statement is just the fact that $(x_0,y_0) \subset 
					F_{j}$.
					By using Corollary \ref{Mean_Val}, we find that
			  	\begin{align*}
			  		\|y_0-x_0\| \, S(R) & \geq
			  		\Phi_j(y_0)-\Phi_j(x_0)
			  		\\
			  		& \geq 
			  		V(R)-2 \, \alpha/U.
			  	\end{align*}
				\end{proof}

				We reach now the main point of this section, which is an upper 
				bound for the probability that infinitely many positive sampling
				events take place.
				\begin{Proposition}
				\label{UBoundProb}
					\begin{align}
						\nonumber
						& \mathbb{P}
						\big( \varepsilon_k=1 \text{ infinitely often} \big)
						\\
						\leq
						\, &
						\sum_{l=0}^{\infty}
						\mathbb{P}
						\big( \bigcap_{j=l}^{\infty}
						\{C_{j+1} \notin F_j,|F_j|\geq \psi\}
						\big)
					\end{align}
				\end{Proposition}
				\begin{proof}
					As $\tau_{\alpha}<\infty$ a.s.,

					\begin{align}
						\nonumber
						& \mathbb{P}
						\big( \varepsilon_k=1 \text{ i.o.} \big)
						\\
						\nonumber
						=
						\, &
						\mathbb{P}
						\big( \{\tau_{\alpha}<\infty\} \cap
						\{ \varepsilon_k=1 \text{ i.o.} \} \big)
						\\
						=
						\nonumber
						\, &
						\sum_{n=0}^{\infty}
						\mathbb{P}
						\big( \{\tau_{\alpha}=n\} \cap
						\{ \varepsilon_k=1 \text{ i.o.} \} \big)
						\\
						\label{AvDern}
						\leq
						\, &
						\sum_{n=0}^{\infty}
						\mathbb{P}
						\big( \{n \geq \tau_{\alpha}\} \cap
						\{ \varepsilon_k=1 \text{ i.o.} \} \big)
					\end{align}

					We can write $\{n \geq \tau_{\alpha}\}=
					\bigcap_{j=n}^{\infty}\{j \geq \tau_{\alpha}\}$
					Using this in (\ref{Include_2}), we obtain
					\begin{align*}
						& \{n \geq \tau_{\alpha}\} \cap 
						\bigcap_{j=n}^{\infty}
						\{\sup \Phi_j > V(R)-\alpha /U \}
						\\
						\subset
						\ &
						\bigcap_{j=n}^{\infty}
						\{C_{j+1} \notin F_j,|F_j|\geq \psi\}.
					\end{align*}

					Combining this with (\ref{Include_1}), we have
					
					\begin{align}
						\nonumber
						& \{n \geq \tau_{\alpha}\} \cap 
						\{\varepsilon_k=1 \text{ i.o.} \}
						\\
						\label{JustForb}
						\subset
						\ &
						\bigcap_{j=n}^{\infty}
						\{C_{j+1} \notin F_j,|F_j|\geq \psi\},
					\end{align}
					and the result follows.
					
				\end{proof}

		\subsection{Finitely many positive sampling events}

			\label{Finit_Sampl}
				
				\begin{Proposition}
				\label{Prop_Finit_Many_Sampl}
					\begin{align}
						\nonumber
						\mathbb{P}
						\big( \varepsilon_k=1 \text{ infinitely often} \big)
						=0.
					\end{align}
				\end{Proposition}

				\begin{proof}
					Using Proposition \ref{UBoundProb}, we simply need to prove 
					that for every $l \geq 0$,
					\begin{align*}
						\mathbb{P} \big(
							\bigcap_{j=l}^{\infty} \{C_{j+1} \notin F_j,
							|F_j|\geq \psi
							\}
						\big)=0.	
					\end{align*}
					By monotone convergence, we have
					\begin{align*}
						& \,
						\mathbb{P} \big(
							\bigcap_{j=l}^{\infty} \{C_{j+1} \notin F_j,
							|F_j|\geq \psi
							\}
						\big)
						=
						\lim_{n \rightarrow \infty}
						\mathbb{P} \big(
							\bigcap_{j=l}^{n} \{C_{j+1} \notin F_j,
							|F_j|\geq \psi
							\}
						\big).		
					\end{align*}
					We are going to work in the slightly more general setting where $Y_0$, and therefore $\Delta_0$, are allowed to be random.
					This is easily dealt with, because we begin by conditioning on
					$\Delta_0$:
					\begin{align*}
						\mathbb{P} \big(
							\bigcap_{j=l}^{n} \{C_{j+1} \notin F_j,
							|F_j|\geq \psi \}
						\big)
						=
						\mathbb{E} \, \big[ 
						\mathbb{P} \big(
							\bigcap_{j=l}^{n} \{C_{j+1} \notin F_j,
							|F_j|\geq \psi \}
						\, | \, \Delta_0
						\big)
						 \, \big].
          \end{align*}
          We then condition on all but the last reproduction events:
					\begin{align}
						\nonumber
						& \,
						\mathbb{P} \Big(
							\bigcap_{j=l}^{n} \{C_{j+1} \notin F_j,
							|F_j|\geq \psi \}
						\, | \, \Delta_0
						\Big)
						\\
						\label{Recurr_Prob}
						= & \,
						\mathbb{E} \, \bigg[
						\mathds{1}_{
						\displaystyle \{ C_{l+1} \notin F_l,|F_l|\geq \psi \} }
						\dots
						\mathds{1}_{
						\displaystyle \{ C_{n} \notin F_{n-1},|F_{n-1}|\geq \psi \} }
						\\
						\nonumber
						& \,
						\quad \quad \quad \quad 
						\mathbb{P} \Big(
							 C_{n+1} \notin F_n,
							|F_n|\geq \psi
						\, | \, \Delta_0 , F_l, C_{l+1} , ..., F_{n-1}, C_{n} 
						\Big)
						\, | \, \Delta_0
						\bigg].
					\end{align}											
					We can calculate the last term by conditioning on $F_n$ and 
					$\Delta_n$:
					\begin{align}
						\nonumber
						& \,
						\mathbb{P} \big(
							 C_{n+1} \notin F_n,
							|F_n|\geq \psi
						\, | \, \Delta_0 , F_l, C_{l+1} , ..., F_{n-1}, C_{n} , F_n, 
						\Delta_n
						\big)
						\\[10pt]
						\nonumber
						= & \,
						\mathds{1}_{|F_n|\geq \psi}
						\mathbb{P} \big(
							 C_{n+1} \notin F_n
						\, | \, F_n, \Delta_n
						\big)				
						\\[10pt]
						\nonumber
						= & \,
						\mathds{1}_{|F_n|\geq \psi}
						\Big( 1- \frac{|F_n|}{|\Delta_n^R|}
						\Big)
						\\[1pt]
						\nonumber
						\leq & \,
						1- \frac{\psi}{|\Delta_n^R|}
						\\[1pt]
						\label{Ineq_Prob}
						\leq & \,
						1- \frac{\psi}{|\Delta_0^R|+n V(2R)}.
					\end{align}					
					The second and third equalities come from the fact that 
					conditionally on $\Delta_n$, $C_{n+1}$ is sampled uniformly
					from $\Delta_n^R$, independently of the past. The last 
					inequality comes from (\ref{Dyn_Clust}):
	        \begin{equation*}
	        	\Delta_{n}=\Delta_0 \cup  
	        	\bigcup_{
	        		\substack{
	        			1\leq k \leq n,\\
	        			\varepsilon_{k}=1	}
	        		}
	        	 B(C_{k},R).
	        \end{equation*}
	        In particular, it implies that	        
	        \begin{align*}
	        	|\Delta_n^R|
	        	\leq &
	        	|\Delta_0^R|
	        	+ |B(C_1,R)^R|
	        	+ \dots
	        	+ |B(C_n,R)^R|
	        	\\
	        	\leq &
	        	|\Delta_0^R|
	        	+ n \, V(2R).
	        \end{align*}
					Putting inequality (\ref{Ineq_Prob}) into (\ref{Recurr_Prob}),
					we obtain the following upper bound:

					\begin{align}
						\nonumber
						& \,
						\mathbb{P} \Big(
							\bigcap_{j=l}^{n} \{C_{j+1} \notin F_j,
							|F_j|\geq \psi \}
						\, | \, \Delta_0
						\Big)
						\\
						\label{Recurr_Prob_2}
						\leq & \,
						\Big(
						1- \frac{\psi}{|\Delta_0^R|+n V(2R)}
						\Big)
						\mathbb{P} \Big(
							\bigcap_{j=l}^{n-1} \{C_{j+1} \notin F_j,
							|F_j|\geq \psi \}
						\, | \, \Delta_0
						\Big).
					\end{align}						
					Inequality (\ref{Recurr_Prob_2}) provides a recurrence 
					relation, which we can solve immediately to obtain
					\begin{align*}
						\mathbb{P} \big(
							\bigcap_{j=l}^{n} \{C_{j+1} \notin F_j,
							|F_j|\geq \psi \}
						\, | \, \Delta_0
						\big)
						\leq \,
						\prod_{j=l}^n
						\Big(1- \frac{\psi}{|\Delta_0^R|+j V(2R)}\Big) \,.
					\end{align*}
					Taking expectation and then the limit as $n\rightarrow \infty$ 
					, we obtain
					\begin{align*}
						\mathbb{P} \big(
							\bigcap_{j=l}^{\infty} \{C_{j+1} \notin F_j,
							|F_j|\geq \psi \}
						\big)
						\leq & \,
						\lim_{n \rightarrow \infty}
						\mathbb{E} \, \big[
						\prod_{j=l}^n
						\big(1- \frac{\psi}{|\Delta_0^R|+j V(2R)}\big) \,
						\big] 
						\\
						\leq & \,
						\mathbb{E} \, \big[
						\prod_{j=l}^\infty
						\big(1- \frac{\psi}{|\Delta_0^R|+j V(2R)}\big) \,
						\big].
					\end{align*}
					We rewrite the infinite random product using logarithms:
					\begin{align*}
						& \,
						\prod_{j=l}^\infty
						\big(1- \frac{\psi}{|\Delta_0^R|+j V(2R)}\big)
						\\
						= & \,
						\exp \Big(
						\sum_{j=l}^{\infty}
						\log
						\big(1- \frac{\psi}{|\Delta_0^R|+j V(2R)}\big)
						\Big).
					\end{align*}
					After observing that
					\begin{align*}
						& \,
						\log
						\big(1- \frac{\psi}{|\Delta_0^R|+j V(2R)}\big)						
						\mathop{\sim}_{j \rightarrow \infty}^{a.s.} \,
						\frac{-\psi/V(2R)}{j},
					\end{align*}
					we conclude that the infinite product is almost surely equal 
					to $0$. Because we chose $Y_0$ to be deterministic, we conclude that
					\begin{align*}
						\mathbb{P} \big(
							\bigcap_{j=l}^{\infty} \{C_{j+1} \notin F_j,
							|F_j|\geq \psi \}
						\big)
						=0.
					\end{align*}
				\end{proof}
				\begin{Remark}
				\label{Rmk_Delta0_Random}
					In the case where we take $Y_0$ to be random, a 
					sufficient condition for 
					the expectation of the infinite product to also be 
					equal to $0$ is simply $\E(|\Delta_0|)<\infty$, that is the volume of the initial support has a finite expectation.
				\end{Remark}

	\section{Proof of the theorems}
	\label{Prf_Thrm}	
	
		\subsection{Proof of Proposition \ref{Conv_Delta_n}}

				We proved in Proposition \ref{Prop_Finit_Many_Sampl} that with 
				probability one,
				there are only finitely many sampling events. This means that there 
				exists an almost surely finite 
				random time $\kappa$  such that
				\begin{align}
					\forall n > \kappa, \quad \varepsilon_n=0.
				\end{align}
				We recall the dynamics of the cluster $\Delta_n$ described by 
				(\ref{Dyn_Clust}):
        \begin{equation*}
        	\Delta_{n}=\Delta_0 \cup  
        	\bigcup_{
        		\substack{
        			1\leq k \leq n,\\
        			\varepsilon_{k}=1	}
        		}
        	 B(C_{k},R).
        \end{equation*}
        Therefore, if we define $B:=\Delta_{\kappa}$, we have
				\begin{align}
					\forall n > \kappa, \quad \Delta_n=B,
				\end{align}
				and the proof of Proposition \ref{Conv_Delta_n} is complete.\\

				We can now generalise Proposition \ref{Conv_Delta_n} 
				by removing the technical condition on the starting 
				point and allowing $Y_0$
				to be any function in $\mathcal{S}_c$.
				\begin{Proposition}
					\label{Conv_Delta_n_General}
					Suppose $Y_0=f \in \mathcal{S}_c$.
					Then, there exists an almost surely finite 
					random time $\kappa$  such that
					\begin{align}
						\forall n > \kappa, \quad \varepsilon_n=0.
					\end{align}
					Therefore, there exists an
					almost surely bounded
					random set $B \in \mathbb{R}^d$  such that
					\begin{align}
						\forall n > \kappa, \quad \Delta_n=B.
					\end{align}
				\end{Proposition}
        \begin{proof}
					We proceed by coupling $Y$ with a Markov chain $\widetilde{Y}$
					with the same transition probabilities,
					but started from $\widetilde{Y}_0=\delta_{B(C_0,r_0)}$ such that
					$Y_0 \leq \widetilde{Y}_0$.
					We denote the initial 	conditions  by
					$\widetilde{Y}_0=\widetilde{f}$ and $Y_0=f$. We first build 
					$\widetilde{Y}$ as 
					described in
					Definition \ref{Def_Yn}. We then use the sequences 
					$(\widetilde{C}_n)_{n 
					\geq 1}$ and	$(\widetilde{V}_n)_{n \geq 1}$ that we used 
					to construct $\widetilde{Y}$  in the following way. First consider the 
					random sequence $Y'$
					defined by $Y'_0=f$, and for $n \geq 0$,
				  \begin{align*}
				  	Y'_{n+1}=Y'_n+U \, \delta_{B(\widetilde{C}_{n+1},R)} \, 
				  	\bigl( \mathds{1}_{ \{ 
				  	\widetilde{V}_{n+1} \leq 	Y'_n(\widetilde{C}_{n+1})\} } -Y'_n	\bigr).			  	
				  \end{align*}
				  We can prove by induction that
					\begin{equation}
						\label{Monoton}
						\forall n \geq 0, \, Y'_n \leq \widetilde{Y}_n.
					\end{equation}
					It is of course true at $n=0$, and then we just 
					observe that if $Y'_n \leq \widetilde{Y}_n$, then
				  \begin{align*}
				  	& \widetilde{Y}_{n+1}-{Y}_{n+1}'
				  	\\
				  	=&
				  	\big(1-U	\, \delta_{B(\widetilde{C}_{n+1},R)} \big)
				  	\big(\widetilde{Y}_n-{Y}_n'\big)
				  	\\
				  	&  \quad \quad \quad \quad 
				  	+ U	\, \delta_{B(\widetilde{C}_{n+1},R)} \, 
				  	\bigl( 
				  		\mathds{1}_{ \{ \widetilde{V}_{n+1} \leq 	
				  		\widetilde{Y}_n(\widetilde{C}_{n+1})\} } 
				  	- \mathds{1}_{ \{ \widetilde{V}_{n+1} \leq 	
				  	{Y}_n'(\widetilde{C}_{n+1})\} } 
				  	\bigr)
				  	\\
				  	\geq & 0.
				  \end{align*}

				  We denote by 
				  $\widetilde{\Delta}$ and 
				  $\Delta'$
				  the respective sequences of supports, and in particular we have 
				  proved that
					\begin{equation}
						\forall n \geq 0, \, \Delta'_n \subset \widetilde{\Delta}_n.
					\end{equation}
				  
				  We define 
				  the sequence of $\big(\sigma(\widetilde{\mathcal{P}}_n,f)\big)_{n 
				  \geq 0}$-stopping 
				  times $(J_n)_{n \geq 0}$ by setting
					\begin{align}
						\label{Incl_Deltas}
						\left\{
						\begin{aligned}
							& J_0=0, 
							\\
							& J_{n+1}=\inf \{k > J_n: \widetilde{C}_k \in 
							{\Big(\Delta'_{k-1}}\Big)^{R}\}.
						\end{aligned}
						\right.
					\end{align}
					We now construct $(Y_n)_{n \geq 0}$, 
					$(C_n)_{n \geq 1}$
					and $(V_n)_{n \geq 1}$  by taking
					\begin{align*}
						\left\{
						\begin{aligned}
							& Y_n:=Y'_{J_n}, \, n \geq 0
							\\
							& C_n:=\widetilde{C}_{J_n}, \, n \geq 1
							\\
							& V_n:=\widetilde{V}_{J_n}, \, n \geq 1.
						\end{aligned}
						\right.
					\end{align*}
					We denote by $\Delta_n$ the support of 
					$Y_n$, and we define
					the filtration $(\mathcal{P}_n)_{n \geq 0}$ to be
					\begin{align*}
						\left\{
						\begin{aligned}
							&
							\mathcal{P}_0:=\sigma(Y_0),\\
							&
							\mathcal{P}_n:=
							\sigma(C_{1},\dots,C_{n},
							V_{1},\dots,V_{n},
							Y_0,\dots,Y_n),
						\end{aligned}
						\right.
					\end{align*}
					By construction, conditionally on		
					$\mathcal{P}_{n}$,	$C_{n+1}$ is distributed 
					uniformly on
					$\Delta_{n}^{R}$. Because $V_{n+1}$ is 
					independent of
					$\mathcal{P}_{n}$,  
					we conclude that the law of $Y$ is the one given in 
					Definition
					\ref{Def_Yn}.
					
					Using (\ref{Monoton}), we see that
					\begin{align}
					\label{Comp_Y_Ytilde}
						Y_n=Y'_{J_n} \leq \widetilde{Y}_{J_n}.
					\end{align}
					We introduce
					\begin{align*}
						\left\{
						\begin{aligned}
							&
							\widetilde{\varepsilon}_{n+1}:=\1_{\widetilde{Y}_n(\widetilde{C}_{n+1}) \geq 
							\widetilde{V}_{n+1}},
							\\
							&
							\varepsilon_{n+1}:=\1_{Y_n(C_{n+1}) \geq 
							V_{n+1}}.
						\end{aligned}
						\right.
					\end{align*}
					Because $\widetilde{f}=a\, \delta_{B(C_0,r_0)}$, we can use Proposition 
					\ref{Conv_Delta_n}, and 
					we obtain that there exists $\widetilde{\kappa}$ almost surely finite such that
					\begin{align}
						\forall n > \widetilde{\kappa}, \quad \widetilde{\varepsilon}_n=0.
					\end{align}
					In particular, this implies that there exists 
					$\kappa$ almost surely finite such that
					\begin{align}
						\forall n > \kappa, \quad \widetilde{\varepsilon}_{J_n}=0.
					\end{align}
					Combined with (\ref{Comp_Y_Ytilde}), this implies that
					\begin{align}
						\forall n > \kappa, \quad \varepsilon_{n}=0,
					\end{align}
					and the conclusion follows.
        \end{proof}
        
		\subsection{The continuous time process is non explosive}

				We are now going to construct explicitly the process $(X_t)_{t \geq 0}$
				with generator (\ref{Gener_ModSLFV_Red}) as a continous time Markov chain,
				by using $(Y_n)_{n \geq 0}$ as the embedded Markov chain.
				
				\begin{Definition}
					Consider an i.i.d sequence 
					$(E_1,E_2,\dots )$ of Exp(1) random variables. We define the 
					\textsl{jump times}
					$(T_0,T_1,\dots )$ by setting $T_0=0$ and
					\begin{align}
					\label{Def_Times_Tn}
							T_n=\frac{E_1}{\lambda(Y_0)}+ 
							\dots+\frac{E_n}{\lambda(Y_{n-1})}, \quad n \geq 1,
					\end{align}
					where $\lambda(f):=|\text{(Supp}(f))^R|$ for $f \in \mathcal{S}_c$.\\
					We can then define a stochastic process 
					$(X_t)_{t\geq0}$ by setting 
					\begin{align}
					\label{Def_Xt}
						\forall n \geq 0, \, \forall t \in [T_n,T_{n+1}),
						\quad 
						X_t&=Y_n.
					\end{align}
				\end{Definition}
				We recall that the set $\text{(Supp}(f))^R$ is the $R$-expansion
				of the support of $f$, that is the set of points at a distance less than $R$
				from the support of $f$. The quantity $\lambda(f)$ is its volume, and 
				it is the rate at which the process $(X_t)_{t \geq 0}$ jumps out of the state $f$. 
				This will be verified in the following proposition
				by checking that we have the correct generator.
				
				\begin{Proposition}
				\label{CTMC}
					The process
					$(X_t)_{t\geq0}$ constructed in (\ref{Def_Xt}) is a non-explosive 
					$\mathcal{S}_c$-valued continuous time Markov chain.
					Moreover, its generator is given by (\ref{Gener_ModSLFV_Red}).
				\end{Proposition}

				\begin{proof}
					The first thing to verify is that $X_t$ is really defined for 
					all nonnegative $t$.
					This is equivalent to saying that 
					\begin{align*}
						\mathbb{P} \big[ \,
						T_n \xrightarrow[]{} \infty 
						\text{ as }
						n \rightarrow \infty
						\, \big]=1,
					\end{align*}
					that is
					\begin{align*}
						\mathbb{P}\big[ \, 
						\sum_{n=1}^{\infty} \frac{E_n}{\lambda(Y_{n-1})}
						=\infty
						\, \big]=1.
					\end{align*}
					We show in Proposition \ref{Conv_Delta_n}
					that $(\Delta_n)_{n \geq 0}$ converges in a 
					finite number of steps to a bounded set $B$. This means that 
					almost 
					surely,
					there is a random time $\kappa$ such that for all $n \geq 
					\kappa$,
					\begin{align*}
						\lambda(Y_n)=
							|\text{B}^R|
					\end{align*}
					which implies that
					\begin{align*}
						\sum_{n=\kappa+1}^{\infty} 
						\frac{E_{n}}{\lambda(Y_{n-1})}
						=  \ 
						\frac{\displaystyle \sum_{n={\kappa+1}}^{\infty} 
						E_n}{\displaystyle \,
						|\text{B}^R|}
						=  \ \infty \text{ a.s}.
					\end{align*}
					Hence $(X_t)_{t \geq 0}$ is a stochastic process 
					defined for all $t 
					\geq 0$.
					The Markov property is obvious, and this shows that $(X_t)_{t 
					\geq 0}$
					is a non-explosive continuous time Markov chain.
					We can then write the generator of $(X_t)_{t\geq 0}$ for functions
					$G: \mathcal{S}_c \mapsto \R$ as
					\begin{align*}
						\L G(f)=
						\int_{(\text{Supp}(f))^R}
						\int_0^1 G \big[ 
						f+U \delta_{B(c,R)} (\1_{v \leq f(c)} - f)
						\big]
						-G(f)
						\, dv \, dc.
					\end{align*}
					If we take $G=I_n(\cdot,\psi)$ as defined in (\ref{Test_Func_Red}), the generator
					of $(X_t)_{t\geq0}$ takes the form (\ref{Gener_ModSLFV_Red}).
				\end{proof}
		
		\subsection{Proof of Theorem \ref{Main_Result}}
		\label{Proof_Main_Result}
		
				We have seen in Proposition \ref{CTMC}
				that the process $(X_t)_{t\geq0}$ is a non explosive 
				continuous 
				time Markov chain. Therefore, the trajectories of $(X_t)_{t\geq0}$
				are completely described by its embedded Markov chain $(Y_n)_{n \geq 0}$.
				
				In particular, for all $n \geq 0$, for all $t \in [T_n,T_{n+1})$, 
				$\text{Supp}(X_t)=\Delta_n$.
				Using the result from Proposition \ref{Conv_Delta_n_General},
				and the the sequence of times $(T_n)_{n\geq0}$ defined in (\ref{Def_Times_Tn}),
				there exists a finite random set $B \subset \mathbb{R}^d$, and an
				almost surely finite
				random time 
				$T:=T_{\kappa}$, such that
				\begin{equation}
					\forall t > T, \,
					\text{Supp}(X_t) = B \quad \text{a.s.}
				\end{equation} 

				The second point to prove here is the extinction of the population.
				From Proposition (\ref{Conv_Delta_n_General}),
				\begin{align}
					\forall n > \kappa, \quad \varepsilon_n=0.
				\end{align}
				This implies that at every point $x$,
				the frequency $(X_t(x))_{t \geq T}$ converges geometrically to zero,
				which concludes the proof.
				
	\section{Conclusion}
		
        Although the S$\Lambda$FV process is constructed in great generality,
        our study was restricted to the case where $R$ and $U$ are  constant.
		In the setting described in \cite{BARTON_ETHERIDGE_VEBER_2010}, these quantities 
		can be made random by adding extra dimensions to the space-time Poisson point process. We then 
		define $\Pi$ on the space $[0,\infty) \times \mathbb{R}^d \times [0,1] \times (0,\infty) \times [0,1] $, 
		with intensity $dt\otimes dc \otimes dv \times \zeta(dr,du)$, such that
		$$
        	\int_{(0,\infty) \times [0,1]} u \, r^d \zeta(dr,du)< \infty.		
		$$
				
        Our result holds in the case $\zeta(dr,du)=\delta_{R,U}(dr,du)$ because the volume of $\Delta_n$ increases
        at most linearly with $n$.         
        We could imagine extending the same result using
        practically the same method to the case
        where
        \begin{align*}
        	\int_{(0,\infty) \times [0,1]} r^d \zeta(dr,du) < \infty,
        \end{align*}
        because the process still jumps at finite rate, and the
        volume of $\Delta_n$ is at most of order $n \E(R)$, where $R$
        is a realisation of the random radius. The problem comes from
        the fact that the radii being random makes the construction
        of the Markov chain more complicated. Morally the result remains true
        in this case, but the proof becomes significantly more involved.\\
        
        The situation where 
        \begin{align*}
        	\int_{(0,\infty) \times [0,1]} r^d \zeta(dr,du) = \infty
        \end{align*}
        is radically different, because now the process jumps at an infinite rate.
        The problem is that we do not have a description of the geometry of
        the process at time $t>0$. The behaviour is not obvious, and it cannot be
        simulated. For us this remains an open question, which would certainly require 
        different techniques.

	\bibliographystyle{plain}
	\bibliography{Biblio_Phd}
	
\end{document}